\documentclass[leqno]{article}

\usepackage{layout}
\usepackage[final]{pdfpages}
\usepackage{tikz}

\usepackage[OT1]{fontenc}
\usepackage{color}
\usepackage{amsthm,amsmath,graphicx,latexsym,amssymb,amscd,amsfonts,enumerate, bm}
\usepackage[colorlinks,citecolor=blue,linkcolor=red,urlcolor=blue]{hyperref}
\usepackage{authblk} 

\theoremstyle{plain}
\newtheorem{theorem}{Theorem}[section]
\newtheorem{proposition}[theorem]{Proposition}
\newtheorem{lemma}[theorem]{Lemma}

\newtheorem{remark}[theorem]{Remark}
\newtheorem{example}[theorem]{Example}

\def\R{\mathbb{R}}
\def\P{\mathbb{P}}
\def\E{\mathbb{E}}
\def\D{\mathbb{D}}
\def\eps{\varepsilon}
\def\ud{\mathrm{d}}

\begin{document}

\title{\textbf{\Large Heat content for Gaussian processes:\\ small-time asymptotic analysis}}

\date{\today}

\author{Kei Kobayashi\thanks{Department of Mathematics, Fordham University. Email:\ kkobayashi5@fordham.edu} \ and Hyunchul Park\thanks{Department of Mathematics, State University of New York at New Paltz, Email:\ parkh@newpaltz.edu}}

\maketitle

\begin{abstract}
This paper establishes the small-time asymptotic behaviors of the regular heat content and spectral heat content for general Gaussian processes in both one-dimensional and multi-dimensional settings, where the boundary of the underlying domain satisfies some smoothness condition. For the amount of heat loss associated with the spectral heat content, the exact asymptotic behavior with the rate function being the expected supremum process is obtained, whereas for the regular heat content, the exact asymptotic behavior is described in terms of the standard deviation function. 
 \vspace{3mm}	

\vspace{3mm}


\end{abstract}

\section{Introduction}\label{section_1}

For a Brownian motion $W=(W_t)_{t\ge 0}$ and a bounded domain $D$ in $\R^d$, consider the function  
\begin{align*}
	Q^W_D(t):=\int_D \P_x(\tau^{W}_D>t)\ud x, \ \ \ t>0, 
\end{align*}
where $\tau^W_D=\inf\{ t>0: W_t\in D^c\}$ denotes the first exit time of $W$ from $D$, and $\P_x$ is the law under which $W$ starts at the point $x\in D$. This function, called the spectral heat content (SHC), measures the amount of heat that has not exited the domain $D$ as of time $t$. The intensive study of the SHC for Brownian motions that was conducted more than three decades ago revealed that the SHC contains geometric information about the domain and spectral information about the infinitesimal generator of the underlying killed Brownian motions; see e.g., \cite{vandenBerg1989,vandenBergGilkey94,vandenBergLeGall94,vandenBerg1990}. 
Associated with the SHC is the regular heat content (RHC) defined by 
\[
	H^W_D(t):=\int_D \P_x(W_t\in D)\ud x, \ \ \ t>0. 
\]
The RHC also measures the amount of heat contained in the domain $D$, but heat particles exiting the domain do not get killed and hence are allowed to return to $D$; see e.g., \cite{Miranda2007,vandenBerg2015} for discussions of the RHC. 
In contrast, for L\'evy processes, the investigation into the RHC and SHC started within the last decade, where the presence of jumps requires careful analysis, especially in a multi-dimensional setting (i.e., $d\ge 2$); see e.g., \cite{Valverde2016,Valverde2017,GPS19,KP22,KP23,KP23-unified,P21,ParkSong19,ParkSong22}. 
The RHC and SHC for processes with delay introduced by inverse subordinators have been studied in \cite{KP22,KP23}.

The purpose of this paper is to introduce the SHC and RHC for general Gaussian processes, including fractional Brownian motions and Ornstein--Uhlenbeck processes, and study their asymptotic behaviors as $t\downarrow 0$. 
We will conduct the study in both one-dimensional and multi-dimensional settings. In the one-dimensional setting, Theorems \ref{thm:GaussianMarkov_regular} and \ref{thm:GaussianMarkov_spectral} establish the exact first-order asymptotic behaviors of the RHC and SHC, respectively, with error bounds of exponential decay obtained. In the multi-dimensional setting, we assume that the domain $D$ is a $C^{1,1}$ open set and that the process has i.i.d.\ Gaussian components; the latter may look somewhat restrictive, but Theorems \ref{thm:GaussianMarkov_regular_high} and \ref{thm:GaussianMarkov_spectral_high} still recover the well-known first-order small-time asymptotic behaviors of the RHC and SHC for Brownian motions, respectively. Separately, Proposition \ref{prop:GaussianMarkov_spectral_upperboundonly} establishes a one-sided bound for the SHC, with an error bound of exponential decay valid for all small enough $t>0$, 
under the weaker assumption that the components of the process are not necessarily i.i.d.

Let us now make three key remarks on our results. First and foremost, as of now, no article can be found in the literature on the SHC and RHC for Gaussian processes, except for the special case of Brownian motions. Therefore, our theorems produce a number of new results as corollaries, including those for fractional Brownian motions and fractional/non-fractional Ornstein--Uhlenbeck processes. These examples are provided in Section \ref{section_examples}.

Second, our statements do not assume self-similarity of the Gaussian process, and to the authors' knowledge, 
even the proofs of Theorems \ref{thm:GaussianMarkov_regular} and \ref{thm:GaussianMarkov_spectral} in the one-dimensional setting are significantly different from the proofs for the Brownian motion case available in the literature including \cite{Valverde2016,Valverde2017,vandenBerg1989}.
On the other hand, the proof of Theorem \ref{thm:GaussianMarkov_spectral_high} employs an \textit{approximate} scaling argument that is similar to the one in \cite{KP23-unified}, which requires an assumption on weak convergence of a scaled Gaussian process in the Skorokhod space. We admit that the latter is a strong condition, but it still covers some non-self-similar Gaussian processes.  

Our third remark is that for the exact asymptotic results in Theorems \ref{thm:GaussianMarkov_regular_high}, \ref{thm:GaussianMarkov_regular}, \ref{thm:GaussianMarkov_spectral} and \ref{thm:GaussianMarkov_spectral_high}, the rate functions for the RHC and SHC are described in terms of the standard deviation function and the expected supremum function of the underlying Gaussian process, respectively. The distinct descriptions for the rate functions reflect the difference in the natures of the RHC and SHC. Namely, the SHC requires tracking of particle locations from the beginning through a given time point $t$ since whether the particles have ever exited the domain or not matters (as they are killed upon exit), whereas the RHC concerns particle locations at a given time $t$ only since particles never get killed.

The organization of this paper is as follows. Section \ref{section_2} reviews known facts and introduces the RHC and SHC for Gaussian processes. 
Sections \ref{section_results_regular} and \ref{section_results_spectral} provide our results for the RHC and SHC, respectively, as well as their proofs, except for the proof of Theorem \ref{thm:GaussianMarkov_spectral_high}, which requires a series of lemmas and technical arguments and is postponed to the final section, Section \ref{proof}. Section \ref{section_examples} illustrates applications of our results to some concrete Gaussian processes.

\section{Preliminaries}\label{section_2}

A stochastic process $X=(X_t)_{t\in[0,1]}$ in $\R$ is called a \textit{Gaussian process} if all its finite-dimensional distributions are multivariate Gaussian. The distribution of a zero-mean Gaussian process $X$ is characterized by its covariance function 
\[
	R_X(s,t):=\textrm{Cov}(X_s,X_t)=\E[X_sX_t], \ \ s,t\in [0,1].
\] 
Throughout the paper, we assume that a given Gaussian process $X$ is defined on the space $\Omega=\D[0,1]$ of c\`adl\`ag functions on $[0,1]$ that is equipped with the Skorokhod $J_1$ topology, with $X_t(\omega)=\omega(t)$ for $\omega\in \D[0,1]$ and $t\in[0,1]$. The Borel $\sigma$-algebra $\mathcal{F}=\mathcal{B}(\D[0,1])$ generated by the $J_1$ topology coincides with the $\sigma$-algebra generated by the coordinate projections (or the finite-dimensional cylinder sets); see the discussion given in \cite[Section 11.5.3]{Whitt}.
 We choose a probability measure $\P$ on $(\Omega,\mathcal{F})=(\D[0,1],\mathcal{B}(\D[0,1]))$ in such a way that the projections $(X_t)_{t\in[0,1]}$ constitute a zero-mean Gaussian process with covariance function $R_X$ such that  $\P(X_0=0)=1$.

For a zero-mean Gaussian process $X$, the variance function and its supremum are denoted by 
\begin{align}\label{0001}
	R_t:=R_X(t,t)=\E[X_t^2] \ \ \ \textrm{and} \ \ \ 
	\sigma^2_t:=\sup_{s\in[0,t]} R_s, \ \ t\in [0,1], 
\end{align}
respectively.
The first moment of the running supremum of $X$ is denoted by 
\[
	\mu_t:=\E\biggl[\sup_{s\in[0,t]}X_s\biggr],
	\ \ t\in [0,1].
\]
For example, if $X$ is a Brownian motion
$W=(W_t)_{t\in [0,1]}$ in $\R$ with characteristic function $\E[e^{iuW_t}]=e^{-tu^2}$, then $R_t=2t$ and by \cite[Problem 2.8.2]{KaratzasShreve},
\begin{align}\label{0006}
	\mu_t
	=\int_0^\infty \frac{x}{\sqrt{\pi t}}e^{-\frac{x^2}{4t}}\ud x 
	=2\sqrt{\frac{t}{\pi}}.
\end{align}

We assume throughout the paper that for each fixed $t\in(0,1]$, the zero-mean Gaussian variable $X_t$ is non-degenerate; i.e., $\P(X_t\ne 0)>0$. This immediately implies $R_t>0$. Moreover, it follows that $\mu_t>0$. Indeed, the assumption that $X_0=0$ yields $\mu_t\ge 0$, whereas if $\mu_t=0$, then $\sup_{s\in[0,t]} X_s=0$ $\P$-a.s., so in particular, $X_t\le 0$ $\P$-a.s.; however, this is impossible since the support of the non-degenerate Gaussian variable $X_t$ is $\R$.
Note also that the latter argument yields $\sup_{s\in[0,t]} X_s>0$ $\P$-a.s., and hence, $\mu_t$ has the integral representation
$
	\mu_t=\int_0^\infty \P\bigl(\sup_{s\in[0,t]} X_s>x\bigr)\ud x.
$

For each $t\in(0,1]$, assume that 
 \begin{align}\label{a.s.boundedness}
 	\sup_{s\in[0,t]}X_s<\infty \ \ \P\textrm{-a.s.},
\end{align}
which is known to be equivalent to the condition that $\mu_t<\infty$. 
In fact, the following Borell--TIS inequality holds (discovered independently by Borell and by Tsirelson, Ibragimov and Sudakov):
\begin{align}\label{Borell-TIS}
	\P\biggl(\sup_{s\in[0,t]} X_s>x\biggr)\le 2 e^{-\frac{(x-\mu_t)^2}{2\sigma_t^2}} \ \ \textrm{for all} \ \ x>\mu_t. 
\end{align}
See e.g., the discussion given in \cite[Section 1]{Samorodnitsky1991}. 
Moreover, $\sigma^2_t<\infty$ due to  \cite[Theorem 2.5]{MarcusShepp1971}.
For stationary Gaussian processes, the almost sure boundedness condition \eqref{a.s.boundedness} is equivalent to continuity of the sample paths; see e.g., \cite[Theorem 10.2]{Lifshits_book2}.

Under the almost sure boundedness condition \eqref{a.s.boundedness}, by the dominated convergence theorem and the right-continuity of the paths of $X$, 
\begin{align}\label{limitmeansup}
	\mu_t\downarrow 0 \ \ \textrm{as} \  \ t\downarrow 0.
\end{align}
The almost sure right-continuity of $X$ implies right-continuity in probability, and since convergence in probability in the space of zero-mean Gaussian variables coincides with convergence in $L^2$ due to \cite[Theorem 1.4]{Jansen}, it follows that the variance function $R_t$ defined in \eqref{0001} is right-continuous. 
In particular,
\begin{align}\label{limitvar}
	R_t\downarrow 0 \ \ \textrm{as}  \ \ t\downarrow 0, \ \ \textrm{or equivalently,} \ \ \sigma^2_t\downarrow 0 \ \ \textrm{as} \  \ t\downarrow 0. 
\end{align}

This paper also investigates a stochastic process $(X_t)_{t\in [0,1]}=(X_t^{1},\ldots,X_t^d)_{t\in [0,1]}$ in $\R^d$ with $d\ge 2$ whose components are i.i.d.\ zero-mean Gaussian processes in $\R$. Such processes are simply referred to as Gaussian processes in $\R^d$ and are considered to be random elements taking values in the space $\D[0,1]$ of $\R^d$-valued c\`adl\`ag functions on $[0,1]$ that is equipped with the $J_1$ topology.

Given a Gaussian process $X=(X_t)_{t\in[0,1]}$ starting at the origin under $\P$ and a bounded open set $D$ in $\R^d$, let $\{\P_x:x\in D\}$ be a family of probability measures on $(\D[0,1],\mathcal{B}(\D[0,1]))$ defined by
\begin{align}\label{familyprobabilitymeasures}
	\P_x(F):=\P(F-x), \ \ F\in \mathcal{B}(\D[0,1]),
\end{align}
where $F-x:=\{\omega\in \D[0,1]: \omega(\cdot)+x\in F\}$. The process $X$ under the law $\P_x$ is a shift of the process $X$ under the law $\P=\P_0$ by $x$.
The random time
\[
	\tau^X_D=\inf\{ t>0: X_t\in D^c\}
\]
represents the first exit time of $X$ from the domain $D$. Since $D$ is an open set and $X$ is right-continuous at 0, it follows that $\tau^X_D>0$ $\P_x$-a.s.\ for any $x\in D$. 
The \textit{spectral heat content} and the \textit{regular heat content} for the Gaussian process $X$ at time $t\in(0,1]$ are defined by  
\begin{align}\label{0053}
	Q^X_D(t):=\int_D \P_x(\tau^{X}_D>t)\ud x \ \ \ \textrm{and} \ \ \ 
	H^{X}_D(t):=\int_D \P_x(X_t\in D)\ud x,
\end{align}
respectively.
Since $\{\tau^{X}_D>t\}\subset \{X_t\in D\}$, it follows that 
\begin{align}\label{Q-and-H}
	Q^X_D(t)\le H^{X}_D(t).
\end{align}
The next lemma confirms that the above definitions of the spectral heat content and regular heat content are well-defined.

\begin{lemma}
Let $X=(X_t)_{t\in[0,1]}$ be a stochastic process in $\R^d$ with c\`adl\`ag paths starting at the origin under the probability measure $\P=\P_0$. Let $D$ be a bounded open set in $\R^d$ and let $F\in \mathcal{B}(\D[0,1])$. 
Let a family of probability measures $\{\P_x:x\in D\}$ be defined as in \eqref{familyprobabilitymeasures}. 
Then the mapping $x\mapsto \P_x(F)$ is $\mathcal{B}(D)$-measurable, and consequently, 
the integrals defining the spectral and regular heat contents in \eqref{0053} are well-defined.
\end{lemma}

\begin{proof}
For a finite-dimensional cylinder set of the form $F=\{\omega\in \D[0,1]: \omega(t_i)\in A_i \ \textrm{for} \ i=0,1,\ldots,k\}$, where $k\in\mathbb{N}$, $0=t_0< t_1<\cdots<t_k=1$, and  $A_i\in \mathcal{B}(\R^d)$ for $i=0,1,\ldots,k$, 
\begin{align*}
	\P_x(F)
	&=\P(x\in A_0, X_{t_1}+x\in A_1, \, \ldots, \, X_{t_k}+x\in A_k)\\
	&=\mathbf{1}_{A_0}(x)\int_{\R^{kd}}\mathbf{1}_{A_1}(z_1+x)\cdots\mathbf{1}_{A_k}(z_k+x) \P(X_{t_1}\in \ud z_1,\ldots,X_{t_k}\in \ud z_k).
\end{align*}
Since the mapping 
$
	(x,z_1,\ldots,z_k)\mapsto \mathbf{1}_{A_0}(x)\mathbf{1}_{A_1}(z_1+x)\cdots\mathbf{1}_{A_k}(z_k+x) 
$
 is nonnegative and $\mathcal{B}(D)\times \mathcal{B}(\R^{kd})$-measurable, it follows as part of the statement of the Fubini theorem that the mapping $x\mapsto \P_x(F)$ is $\mathcal{B}(D)$-measurable.

Now, all the finite-dimensional cylinder sets $F$ of the above form generate the $\sigma$-algebra $\mathcal{B}(\D[0,1])$, and the family of sets $F\in \mathcal{B}(\D[0,1])$ for which the mapping $x\mapsto \P_x(F)$ is $\mathcal{B}(D)$-measurable forms a Dynkin system. Therefore, by the Dynkin system theorem (see e.g., \cite[Chapter 2, Theorem 1.3]{KaratzasShreve}), the mapping $x\mapsto \P_x(F)$ is $\mathcal{B}(D)$-measurable for any set $F\in \mathcal{B}(\D[0,1])$. 

Finally, for a fixed $t\in(0,1]$, since $\{\tau_D^X>t\}=\cap_{s\in\mathbb{Q}\cap [0,t]}\{X_s\in D\}$ and $\{X_s\in D\}\in \mathcal{B}(\D[0,1])$ for each $s\in [0,t]$, 
it follows that $\{\tau_D^X>t\}\in \mathcal{B}(\D[0,1])$. Hence, the mapping $x\mapsto \P_x(\tau^{X}_D>t)$ is $\mathcal{B}(D)$-measurable; consequently, 
the integral defining the spectral heat content in \eqref{0053} is well-defined. 
Similarly, the integral defining the regular heat content is well-defined.
\end{proof}

An open set $D$ in $\R^d$ is said to be $C^{1, 1}$ if 
there exist a localization radius $R_1>0$ and a constant $\Lambda>0$ that satisfy the following condition:  
for every $z\in \partial D$, there exist a $C^{1, 1}$ function $\phi=\phi_z: \R^{d-1}\to \R$ 
and an orthonormal coordinate system $CS_z: y=(y^1,\ldots,y^{d-1},y^d)=(\widetilde y, y^d)$
with origin at $z$ 
such that $\phi(0)=0$, $\nabla \phi(0)=(0, \cdots, 0)$, $\|\nabla\phi\|_\infty\le \Lambda$, 
$|\nabla \phi(x_1)-\nabla \phi(x_2)|\le \Lambda |x_1-x_2|$ for $x_1,x_2\in \R^{d-1}$, and
$B(z, R_1)\cap D=B(z, R_1)\cap \{ y=(\widetilde y, y^d) \mbox{ in } CS_z: y^d>\phi(\widetilde y)\}.$ 
By \cite[Lemma 2.2]{AKSZ2007}, 
an open set $D$ is $C^{1, 1}$ if and only if it satisfies the uniform interior and exterior ball condition; i.e., there exists $R>0$ that satisfies the following condition: 
for every $z\in\partial D$, there exist open balls $B_1$ and $B_2$ of the same radius $R$ such that 
\begin{align}\label{def:interior-exterior}
	B_1\subset D, \ B_2\subset\R^d\setminus\overline{D}, \ \textrm{and} \ \partial B_1\cap\partial B_2=\{z\}.
\end{align}
This paper uses the uniform interior and exterior $R$-ball condition for proving Theorem \ref{thm:GaussianMarkov_spectral_high}, and the above function $\phi$ does not appear in any form. We refer to $R$ as the characteristic radius of the $C^{1,1}$ set.

Given a bounded $C^{1,1}$ domain $D$ with characteristic radius $R$, for each $a\in(0,R/2]$, let
\begin{align}\label{def:Da}
	D_{a}=\{x\in D : \delta_{D}(x)>a\}, \ \ \textrm{where} \ \ \delta_{D}(x)=\inf\{|x-y| : y \in \partial D\}. 
\end{align}
That is, $D_a$ is a region obtained by removing points near the boundary $\partial D$ from $D$. Due to the $C^{1,1}$ condition, $D_a$ is non-empty.
It also follows from \cite[Lemma 6.7]{vandenBerg1989} that 
\begin{align}\label{eqn:vdBD89}
|\partial D|\left(\frac{R-a}{R}\right)^{d-1}\leq |\partial D_{a}|\leq |\partial D|\left(\frac{R}{R-a}\right)^{d-1}, 
\end{align}
where $|\partial D|$ and $|\partial D_a|$ are 
the $(d-1)$-dimensional Lebesgue measures of the boundaries $\partial D$ and $\partial D_a$ of the sets $D$ and $D_a$, respectively.
In particular, 
\begin{align}\label{eqn:vdBD89-b}
	|\partial D_{a}|\leq 2^{d-1}|\partial D| \ \ \textrm{for any} \  a\in(0,R/2].
\end{align}

If $X$ is given by a Brownian motion $W$ with $\E[e^{-i\langle \xi, W_t\rangle}]=e^{-t|\xi|^2}$, then 
for a bounded open interval $D$ in $\R$ or a bounded connected $C^{1,1}$ open set $D$ in $\R^d$ with $d\ge 2$
(see \cite[Theorem 6.2]{vandenBerg1989} and \cite[Theorem 2]{vandenBerg2015}), 
\begin{align}
	\label{0031}\lim_{t\downarrow 0}\frac{|D|-Q_{D}^{W}(t)}{\sqrt{t}}&= \frac{2 |\partial D|}{\sqrt{\pi}};\\
	\label{0032}\lim_{t\downarrow 0}\frac{|D|-H^{W}_D(t)}{\sqrt{t}}&= \frac{|\partial D|}{\sqrt{\pi}},
\end{align}
where $|D|$ denotes the $d$-dimensional Lebesgue measure of $D$, and if $D=(a,b)$ in $\R$, then $|D|=b-a$ and $|\partial D|=|\{a,b\}|=2$.

\section{The short-time behavior of the regular heat content}\label{section_results_regular}

We first derive the asymptotic behavior of the \textit{regular} heat content for general Gaussian processes. 
In short, the amount of heat loss from the domain $D$ behaves like the standard deviation function $\sqrt{R_t}$ in small time.
Note that the result below immediately reduces to \eqref{0032} if $X$ is taken to be a Brownian motion $W$ with $\E[e^{-i\langle \xi, W_t\rangle}]=e^{-t|\xi|^2}$. 
Also recall that  if $D=(a,b)$ in $\R$, then $|D|=b-a$ and $|\partial D|=2$.

\begin{theorem}\label{thm:GaussianMarkov_regular_high}
Let $D$ be a non-empty, bounded interval in $\R$ or a non-empty, bounded, connected $C^{1,1}$ open set in $\R^d$  with $d\ge 2$. 
Let $X=(X_t)_{t\in [0,1]}$ be a stochastic process in $\R^d$ whose components are i.i.d.\ zero-mean Gaussian processes with c\`adl\`ag paths starting at the origin and common variance function $R_t$ under the probability measure $\P=\P_0$ such that $X_t$ is non-degenerate for all $t\in (0,1]$. 
 Then
\[
 \lim_{t\downarrow 0} \frac{|D|-H^X_D(t)}{\sqrt{R_t}}
=\frac{|\partial D|}{\sqrt{2\pi}}.
\]
\end{theorem}

\begin{proof}
Recall that the almost sure right continuity of $X$ implies $R_t\downarrow 0$ as $t\downarrow 0$, as stated in  \eqref{limitvar}.
Note the identity $|D|-H^X_D(t)=\int_D \P_x(X_t\in D^c)\ud x$. Applying this identity to a Brownian motion $W$ with $\E[e^{-i\langle \xi, W_t\rangle}]=e^{-t|\xi|^2}$, we can rewrite the statement in \eqref{0032} as
\[
	\lim_{t\downarrow 0}\frac{1}{\sqrt{t}}\int_D \biggl(\int_{D^c}\dfrac{1}{(4\pi t)^{d/2}}e^{-\frac{|x-y|^2}{4t}}\ud y \biggr)\ud x
	= \frac{|\partial D|}{\sqrt{\pi}}.
\]
This, together with \eqref{limitvar}, yields
\[
	\lim_{t\downarrow 0} \sqrt{\frac{2}{R_t}}\int_D \biggl(\int_{D^c}\dfrac{1}{(2\pi R_t)^{d/2}}e^{-\frac{|x-y|^2}{2R_t}}\ud y \biggr)\ud x
	= \frac{|\partial D|}{\sqrt{\pi}}.
\]
This is the desired result since the components of $X$ are assumed to be i.i.d.\ zero-mean Gaussian processes with common variance function $R_t$.  
\end{proof}

In \cite{Valverde2017}, the author derived the small-time asymptotic behavior of the regular heat content and spectral heat content for stable L\'evy processes. The key idea of the argument was to express a given stable process as a Brownian motion time-changed by an independent stable subordinator and to utilize the self-similarity of the Brownian motion; i.e., $(W_{at})_{t\ge 0} =^{\mathrm{d}} (a^{1/2}W_t)_{t\ge 0}$. However, it turns out that the use of the self-similarity is not essential. To show this, we provide an alternative proof of Theorem \ref{thm:GaussianMarkov_regular_high} in the one-dimensional case that comes with an associated error bound of exponential decay. Namely:

\begin{theorem}\label{thm:GaussianMarkov_regular}
Let $D$ be a non-empty, bounded interval in $\R$. 
Let $X=(X_t)_{t\in [0,1]}$ be a Gaussian process in $\R$ with c\`adl\`ag paths starting at the origin 
and variance function $R_t$ 
under the probability measure $\P=\P_0$ such that 
$X_t$ is non-degenerate for all $t\in(0,1]$. 
 Then for any $t\in[0,1]$,
\[
	|D|-H_D^X(t)=\sqrt{\frac{2R_t}{\pi}}- \mathrm{Error}_{H}(t),
\]
where 
\begin{align}\label{error-bound-H1}
	0<\mathrm{Error}_{H}(t)
	\le \frac{2R_t^{\frac 32}}{|D|^2\sqrt{2\pi}}e^{-\frac{|D|^2}{2R_t}}. 
\end{align}
Consequently,
\begin{align}\label{error-bound-H2}
 \lim_{t\downarrow 0} \frac{|D|-H^X_D(t)}{\sqrt{R_t}}
=\sqrt{\frac{2}{\pi}}.
\end{align}
\end{theorem}

\begin{proof}
Suppose $D=(a,b)$ with $-\infty<a<b<\infty$.
By changing variables via $u:=b-x-c$ with $c:=(b-a)/2$, 
\begin{align*}
	|D|-H^X_{D}(t)
	&=\int_a^b \left(\P_x(X_t\le a)+\P_x(X_t\ge b)\right)\ud x
	=\int_a^b \left(\P(X_t+x\le a)+\P(X_t+x\ge b)\right)\ud x\\
	&=\int_{-c}^c \left(\P(X_t\le u-c)+\P(X_t\ge u+c)\right)\ud u.
\end{align*}
Since $X_t$ has a zero-mean Gaussian distribution, $-X_t$ has the same distribution as $X_t$.  Therefore,
$
	\int_{-c}^c\P(X_t\le u-c)\ud u
	=\int_{-c}^c\P(-X_t\le u-c)\ud u
	=\int_{-c}^c\P(X_t\ge v+c)\ud v,
$
and hence, it follows that 
\[
	|D|-H^X_{D}(t)=2\int_{-c}^c \P(X_t\ge u+c)\ud u
	=2\int_0^{2c} \P(X_t\ge u)\ud u. 
\]
Since the distribution function $u\mapsto \P(Y\le u)$ of any random variable $Y$ is right-continuous, $\P(Y\le u)=\P(Y< u)$ for almost every $u$. Also, the equality $\{X_t>u\}=\{X_t\mathbf{1}_{\{X_t\ge 0\}}>u\}$ holds for $u\ge 0$, so  
\begin{align}\label{regularheat}
	|D|-H^X_{D}(t)
	=2\int_0^{2c} \P(X_t> u)\ud u
	=2\E[X_t\mathbf{1}_{\{X_t\ge 0\}}]-2\int_{2c}^{\infty} \P(X_t> u)\ud u.
\end{align}
The first term is calculated as 
\begin{align}\label{0004}
	\E[X_t\mathbf{1}_{\{X_t\ge 0\}}]
	=\int_0^\infty \frac{x}{\sqrt{2\pi R_t}}e^{-\frac{x^2}{2R_t}}\ud x
	=\sqrt{\frac{R_t}{2\pi}}, 
\end{align}
As for the second term, for any $u\ge x>0$, 
by the estimate 
$\P(Z>z)\le e^{-z^2/2}/(z\sqrt{2\pi})$
valid for a standard Gaussian variable $Z$ and any constant $z>0$ (see e.g., \cite[Theorem 1.2.6]{Durrett-5th}),
\begin{align*}
	\P(X_t> u)
	= \P\biggl(\frac{X_t}{\sqrt{R_t}}> \frac{u}{\sqrt{R_t}}\biggr)
	\le \frac{1}{\sqrt{2\pi}} \frac{\sqrt{R_t}}{u}e^{-\frac{u^2}{2R_t}}
	\le \biggl(\frac{u}{x}\biggr)^2\frac{1}{\sqrt{2\pi}}\frac{\sqrt{R_t}}{u}e^{-\frac{u^2}{2R_t}} 
	=\frac{\sqrt{R_t}}{x^2\sqrt{2\pi}}ue^{-\frac{u^2}{2R_t}}.\notag
\end{align*}
Integrating both sides from $x$ to $\infty$ yields 
$\int_{x}^{\infty}\P(X_t>u)\ud u
\le R_t^{3/2}e^{-x^2/(2R_t)}/({x^2\sqrt{2\pi}}).$
The latter inequality, together with \eqref{regularheat} and \eqref{0004}, yields the error bound in \eqref{error-bound-H1}, from which \eqref{error-bound-H2} follows immediately.
\end{proof}

\begin{remark}\label{remark:regular}
\begin{em}
A simple modification of the above proof leads to a more general statement on the regular heat content for a \textit{symmetric, zero-mean (but not necessarily Gaussian)} process $X$ in $\R$ 
with c\`adl\`ag paths such that $\E[X_t\mathbf{1}_{\{X_t\ge 0\}}]\to 0$ as $t\downarrow 0$.
Indeed, 
if the regular heat content is defined as in \eqref{0053},
and if 
$\int_{b-a}^\infty \P(X_t>u)\ud u=o(\E[X_t\mathbf{1}_{\{X_t\ge 0\}}])$
as $t\downarrow 0,$
then the symmetry of the distribution of $X$ yields expression \eqref{regularheat}, from which it follows that
\begin{align}\label{regular-general}
	\lim_{t\downarrow 0}\frac{|D|-H_{D}^X(t)}{\E[X_t\mathbf{1}_{\{X_t\ge 0\}}]} =2.  
\end{align}
Moreover, if $X$ is additionally assumed to be self-similar with index $H>0$, then 
expression \eqref{regular-general} can be rewritten as
$
\lim_{t\downarrow 0}t^{-H}(|D|-H_{D}^X(t))=2\E[X_1\mathbf{1}_{\{X_1\ge 0\}}]. 
$
In particular, this expression recovers the statement in \cite[Equation (1.10)]{Valverde2016} for a symmetric stable L\'evy process with stability index $\alpha\in(1,2)$, which is self-similar with index $H=1/\alpha$.
\end{em}
\end{remark}

\section{The short-time behavior of the spectral heat content}\label{section_results_spectral}

We now turn our attention to the \textit{spectral} heat content for general Gaussian processes. 
 We first establish an analogue of Theorem \ref{thm:GaussianMarkov_regular} in the one-dimensional case that comes with an associated error bound of exponential decay.

\begin{theorem}\label{thm:GaussianMarkov_spectral}
Let $D$ be a non-empty, bounded interval in $\R$.
Let $X=(X_t)_{t\in[0,1]}$ be a zero-mean Gaussian process in $\R$ with c\`adl\`ag paths starting at the origin under the probability measure $\P=\P_0$ such that 
the almost sure boundedness condition \eqref{a.s.boundedness} holds and $X_t$ is non-degenerate for all $t\in (0,1]$.
Then for any $t\in (0,1]$ satisfying $\mu_t<|D|/2$,  
\[
	|D|-Q_D^X(t)=2\mu_t- \mathrm{Error}_Q(t),
\]
where
\begin{align}\label{error-bound-Q1}
0<\mathrm{Error}_Q(t)
	\le \frac{4\sigma_t^2}{|D|-\mu_t} e^{-\frac 12\bigl(\frac{|D|-\mu_t}{\sigma_t}\bigr)^{2}}
		+\frac{4\sigma_t^2}{\frac 12|D|-\mu_t} e^{-\frac 12\bigl(\frac{\frac 12|D|-\mu_t}{\sigma_t}\bigr)^{2}}.
\end{align}
Consequently, 
\begin{align}\label{conclusion_spectral}
	\lim_{t\downarrow 0}\frac{|D|-Q_D^X(t)}{\mu_t} =2.  
\end{align}
\end{theorem}

\begin{proof}
Suppose $D=(a,b)$ with $-\infty<a<b<\infty$.
Observe that
\begin{align*}
	\P_x(\tau^{X}_D>t)
	=\P_x(a<X_s<b \ \textrm{for all} \ s\le t)
	=\P(a-x<X_s<b-x \ \textrm{for all} \ s\le t).
\end{align*}
Since $|D|-Q^{X}_D(t)=\int_a^b \P_x(\tau^X_D\le t)\ud x$, 
\begin{align*}
	|D|-Q^{X}_D(t)
	&=\int_a^b \P(\sup_{0\le s\le t} X_s\ge b-x)\ud x
		+\int_a^b \P(\inf_{0\le s\le t} X_s\le a-x)\ud x\\
	& \ \  \ \	-\int_a^b \P(\sup_{0\le s\le t} X_s\ge b-x,\, \inf_{0\le s\le t} X_s\le a-x)\ud x\\
	&=:I_1(t)+I_2(t)-I_3(t). 
\end{align*}
Since the distribution function $u\mapsto \P(Y\le u)$ of any random variable $Y$ is right-continuous, $\P(Y\le u)=\P(Y< u)$ for almost every $u$ with respect to the Lebesgue measure, and in particular, $\P(\sup_{0\le s\le t} X_s\ge u)=\P(\sup_{0\le s\le t} X_s> u)$ for almost every $u$ as well.
Thus, a simple change of variables yields
\[
	I_1(t)
	=\int_0^{b-a}\P(\sup_{0\le s\le t} X_s> u)\ud u
	=\mu_t-I_4(t),
\]
where
\[ 
	I_4(t):=\int_{b-a}^\infty \P(\sup_{0\le s\le t} X_s> u)\ud u.
\] 
On the other hand, the process $-X=(-X_t)_{t\ge 0}$ is Gaussian and has the same covariance function as $X$, so $-X$ and $X$ have the same distribution; therefore, 
\[
	I_2(t)
	=\int_a^b \P(\inf_{0\le s\le t} (-X_s)\le a-x)\ud x
	=\int_a^b \P(\sup_{0\le s\le t} X_s\ge x-a)\ud x
	=I_1(t).
\]
Thus, we have obtained
\begin{align}\label{estimate-01}
	|D|-Q^{X}_D(t)=2\mu_t-2I_4(t)-I_3(t).
\end{align}

With \eqref{limitmeansup} in mind, take $t>0$ small enough so that 
$\mu_t<c:=(b-a)/2$.
Let
\begin{align}\label{estimate-02}
	\mathrm{Error}_Q(t):=2I_4(t)+I_3(t).
\end{align}
We now show that $\mathrm{Error}_Q(t)$ has the upper bound appearing in \eqref{error-bound-Q1}.
As for $I_4(t)$, 
by the Borell--TIS inequality \eqref{Borell-TIS} and the inequality 
$\P(Z>z)\le e^{-z^2/2}/(z\sqrt{2\pi})$ valid for a standard Gaussian variable $Z$ and any constant $z>0$ (see e.g., \cite[Theorem 1.2.6]{Durrett-5th}), 
\begin{align}\label{estimate-I4}
	I_4(t)
	&\le 2\int_{2c}^\infty e^{-\frac{(u-\mu_t)^2}{2\sigma^2_t}}\ud u 
	= 2 \sigma_t\int_{\frac{2c-\mu_t}{\sigma_t}}^\infty e^{-z^2/2}\ud z
	\le \frac{2\sigma_t^2}{2c-\mu_t} e^{-\frac 12\bigl(\frac{2c-\mu_t}{\sigma_t}\bigr)^{2}}.  
\end{align}
On the other hand, in terms of $I_3(t)$, by the change of variables $u=b-x-c$, 
\[
	I_3(t)=\int_{-c}^c f_{c,t}(u)\ud u, \ \ \textrm{where} \ \ 
	f_{c,t}(u)
	=\P(\sup_{0\le s\le t}X_s\ge u+c,\, \inf_{0\le s\le t} X_s\le u-c). 
\] 
Since $-X$ is another Gaussian process having the same distribution as $X$, 
\begin{align*}
	f_{c,t}(u)
	&=\P(\sup_{0\le s\le t}(-X_s)\ge u+c,\, \inf_{0\le s\le t} (-X_s)\le u-c)
	=\P(-\inf_{0\le s\le t}X_s\ge u+c,\, -\sup_{0\le s\le t} X_s\le u-c)\\
	&=\P(\inf_{0\le s\le t}X_s\le -u-c,\, \sup_{0\le s\le t} X_s\ge -u+c)
	=f_{c,t}(-u).
\end{align*}
Hence, $f_{c,t}$ is an even function, and therefore,    
\begin{align}\label{estimate-04}
	I_3(t)
	=2\int_0^c f_{c,t}(u)\ud u
	\le 2\int_0^c \P(\sup_{0\le s\le t}X_s\ge u+c)\ud u
	\le 2\int_c^{\infty} \P(\sup_{0\le s\le t}X_s\ge v)\ud v. 
\end{align}
By \eqref{estimate-04} and part of the estimate in \eqref{estimate-I4} (with $2c$ replaced by $c$), 
\begin{align}\label{estimate-I3}
	I_3(t)
	\le \frac{4\sigma_t^2}{c-\mu_t} e^{-\frac 12\bigl(\frac{c-\mu_t}{\sigma_t}\bigr)^{2}}.
\end{align}
Combining \eqref{estimate-02}, \eqref{estimate-I4} and \eqref{estimate-I3}
gives the bound in \eqref{error-bound-Q1} valid for any small $t>0$ satisfying $\mu_t<c$.

Finally, to obtain \eqref{conclusion_spectral} from \eqref{error-bound-Q1}, we must verify that $\mathrm{Error}_Q(t)=o(\mu_t)$ as $t\downarrow 0$. To this end, observe first that the error bound in \eqref{error-bound-Q1} gives \begin{align}\label{error-bound-Q2}
	\mathrm{Error}_Q(t)\le C \sigma_t^2
\end{align}
for all small enough $t>0$, where $C>0$ is some constant independent of $t$. 
Next, for each small $t>0$, take $s_\ast(t)\in(0,t]$ such that 
$
	\sigma^2_t=\sup_{s\in[0,t]}R_s<2R_{s_\ast(t)}.
$
(For each $t>0$, such a number $s_\ast(t)$ exists since, otherwise, $\sigma^2_t\ge 2R_{s}$ for all $s\in(0,t]$, which would give $\sigma^2_t\ge 2\sigma^2_t$, and hence, $\sigma^2_t=0$, which is impossible since $X_t$ is assumed non-degenerate for all $t>0$.) 
Recall Sudakov's lower bound (see \cite[Theorem 10.5]{Lifshits_book2} and its proof, or \cite[Proposition 1.3]{Mishura2017_ExpectedMaxima}),
which states that for any zero-mean Gaussian random vector $(X_1,\ldots,X_n)$, 
\[
	\E\Bigl[\max_{1\le i\le n} X_i\Bigr]\ge c_n\min_{i\ne j} d(i,j),  
	\ \ \textrm{where} \ \  d(i,j):=\bigl(\E[|X_i-X_j|^2]\bigr)^{1/2}
\]
for some constant $c_n>0$. 
Applying this statement to the vector $(X_0,X_{s_\ast(t)})$ with $X_0=0$ yields
\begin{align*}
	\mu_t
	\ge \E\bigl[\max\{X_0,X_{s_\ast(t)}\}\bigr]
	\ge c_2(R_{s_\ast(t)})^{1/2}.
\end{align*}
Hence,  
\begin{align}\label{sigma-mu-relation}
	\frac{\sigma^2_t}{\mu_t}
	\le \frac{2R_{s_\ast(t)}}{c_2(R_{s_\ast(t)})^{1/2}}
	=\frac{2}{c_2}(R_{s_\ast(t)})^{1/2}, 
\end{align}
which approaches 0 since $s_\ast(t)\to 0$ as $t\downarrow 0$ due to the condition $s_\ast(t)\in(0, t]$.
The latter, together with \eqref{error-bound-Q2}, yields $\mathrm{Error}_Q(t)=o(\mu_t)$, as desired.
\end{proof}

\begin{remark}
\begin{em}
1) If $X$ is given by a Brownian motion $W$ on $\R$ with $\E_0[e^{-i \xi W_t}]=e^{-t\xi^2}$, then $R_t=\sigma^2_t=2t$ and $\mu_t$ is given by  \eqref{0006}, which implies that $\mathrm{Error}_Q(t)\to 0$ as $t\downarrow 0$.
Consequently, 
Theorem \ref{thm:GaussianMarkov_spectral} recovers the statement in \eqref{0031} when $d=1$. 

2) Theorem \ref{thm:GaussianMarkov_spectral} does not require that the Gaussian process $X$ be self-similar. 
In fact, the argument given in the above proof is significantly different from that given for a Brownian motion in \cite{Valverde2016}, which relies heavily on the self-similarity of the Brownian motion. 

3) As discussed in Remark \ref{remark:regular} for regular heat content, a simple modification of the above proof leads to a statement on the spectral heat content for 
more general \textit{symmetric, zero-mean (but not necessarily Gaussian)} processes $X$ in $\R$ 
with c\`adl\`ag paths and $\mu_t<\infty$ for all $t\in(0,1]$.
Indeed, 
suppose that the spectral heat content is defined as in \eqref{0053} and that $\int_{|D|/2}^\infty \P(\sup_{0\le s\le t}X_s>u)\ud u=o(\mu_t)$
as $t\downarrow 0$. 
Then the symmetry of the distribution of $X$ yields \eqref{estimate-01}, and one can also observe that $I_4(t)\le \int_{|D|/2}^\infty \P(\sup_{0\le s\le t}X_s>u)\ud u=o(\mu_t)$ and $I_3(t)=o(\mu_t)$ due to estimate \eqref{estimate-04}. Consequently, $\mathrm{Error}_Q(t)$ defined in \eqref{estimate-02} satisfies $\mathrm{Error}_Q(t)=o(\mu_t)$, from which \eqref{conclusion_spectral} follows.
In particular, under the additional assumption that $X$ is self-similar with index $H>0$,  \eqref{conclusion_spectral} can be rewritten as
\begin{align}\label{selfsimilar-shc}
		\lim_{t\downarrow 0}\frac{|D|-Q_{D}^X(t)}{t^H} =2\mu_1.  
\end{align}
This recovers the statement in \cite[Theorem 1.1(a)]{Valverde2016} for a symmetric stable process with stability index $\alpha\in(1,2)$.
\end{em}
\end{remark}

We now investigate the spectral heat content in a multi-dimensional setting, which requires a careful discussion. 
Our main result in this case is Theorem \ref{thm:GaussianMarkov_spectral_high}, which provides the exact small-time asymptotic behavior of the amount of heat loss $|D|-Q_D^{X}(t)$; however, the statement is given in a somewhat restricted setting. Instead, in Proposition \ref{prop:GaussianMarkov_spectral_upperboundonly} below, we establish an upper bound for the heat loss that is valid for a wide class of Gaussian processes and for any fixed $t>0$ small enough. 

In the remainder of the paper, we use the notation $X^{j}=(X^{j}_t)_{t\in [0,1]}$ to denote the $j$th component of a process $X=(X_t)_{t\in[0,1]}$ in $\R^d$, and $x^{j}$ to denote the $j$th component of any given vector $x\in \R^d$. 
In addition, for each component $X^j$, we write $\mu_{j,t}=\E[\sup_{s\in[0,t]}X_s^{j}]$ and  $\sigma^2_{j,t}=\sup_{s\in[0,t]}\E[(X^j_t)^2]$.
Recall the notations $\delta_D(x)$ and $D_a$ defined in \eqref{def:Da}.

\begin{proposition}\label{prop:GaussianMarkov_spectral_upperboundonly}
Let $D$ be a non-empty, bounded, connected $C^{1,1}$ open set in $\R^d$  with $d\ge 2$ with characteristic radius $R$.
Let $X=(X_t)_{t\in [0,1]}=(X_t^1,\ldots,X_t^d)_{t\in [0,1]}$ be a stochastic process in $\R^d$ whose components are zero-mean Gaussian processes with c\`adl\`ag paths starting at 0, expected running supremums $\mu_{1,t},\ldots,\mu_{d,t}$, and the supremums $\sigma^2_{1,t},\ldots,\sigma^2_{d,t}$ of the variance functions under the probability measure $\P=\P_0$ such that 
the almost sure boundedness condition \eqref{a.s.boundedness} holds and $X_t$ is non-degenerate for all $t\in(0,1]$.
 Then for any $t\in(0,1]$ and $a\in (0,R/2]$ with $\max_{1\le j\le d}\mu_{j,t}<a/\sqrt{d}$, 
 \begin{align}\label{spectral.upper1}
	|D|-Q_D^{X}(t)
	\le d^{3/2} 2^{d}|\partial D|\overline{\mu}_t+4d|D_a| \sum_{j=1}^d 
	\exp\left(-\frac{(\frac{a}{\sqrt{d}}-\mu_{j,t})^2}{2\sigma_{j,t}^2}\right),
 \end{align}
where 
\[
\overline{\mu}_t=\frac{1}{d}\sum_{j=1}^d \mu_{j,t}.
\]
 Consequently,
\begin{align}\label{spectral.upper2}
  \limsup_{t\downarrow 0} \frac{|D|-Q^X_D(t)}{\overline{\mu}_t}
	\le d^{3/2} 2^{d} |\partial D|. 
\end{align}
\end{proposition}

\begin{proof}
Let $B(b,r)$ denote the open ball of radius $r$ centered at $b$. Then for any fixed $x\in D$, since $D\supset B(x,\delta_{D}(x))$, it follows that
$
	\P_x(\tau_D^X\le t)
	\le \P_x(\tau_{B(x,\delta_{D}(x))}^X\le t)
	=\P(\tau_{B(0,\delta_{D}(x))}^X\le t).
$
On the other hand, since $\max_{1\leq j\leq d}|x^{j}|\leq |x|=(\sum_{j=1}^{d}(x^{j})^2)^{1/2}\leq \sqrt{d}\max_{1\leq j\leq d}|x^{j}|$ for any $x=(x^1,\ldots, x^d)\in \R^{d}$, for any fixed $r>0$, 
\begin{align*}
	\P(\tau^X_{B(0,r)}\le t)
	&=\P(\sup_{0\le s\le t} |X_s|\ge r)
	\le \P\biggl(\sqrt{d}\sup_{0\le s\le t} \max_{1\le j\le d} |X^{j}_s|\ge r\biggr)
	\le \P\biggl(\bigcup_{j=1}^d \biggl\{\sup_{0\le s\le t}   |X^{j}_s|\ge \frac{r}{\sqrt{d}}\biggr\}\biggr)\\
	&\le \sum_{j=1}^d \P\biggl(\sup_{0\le s\le t}   |X^{j}_s|\ge \frac{r}{\sqrt{d}}\biggr)
	=2\sum_{j=1}^d  \P\biggl(\sup_{0\le s\le t}   X^{j}_s\ge \frac{r}{\sqrt{d}}\biggr),
\end{align*}
where the last equality follows from the symmetry of the process $X^{j}$. Combining the above two observations yields  
\[
	|D|-Q_D^{X}(t)
	=\int_D \P_x(\tau_D^X\le t)\,\ud x
	\le 2\sum_{j=1}^d  \int_D\P\biggl(\sup_{0\le s\le t}   X^{j}_s\ge \frac{\delta_{D}(x)}{\sqrt{d}}\biggr)\,\ud x.
\]
For any $a\in (0,R/2]$, with \eqref{eqn:vdBD89-b} in mind, split the latter integral into $\int_{D\setminus D_a}$ and $\int_{D_a}$, where $D_a$ is defined in \eqref{def:Da}.
In terms of the first integral,  
\begin{align*}
	I_{1,j}(t)
	&:=\int_{D\setminus D_a}\P\biggl(\sup_{0\le s\le t}   X^{j}_s\ge \frac{\delta_{D}(x)}{\sqrt{d}}\biggr)\,\ud x
	=\int_0^a \P\biggl(\sup_{0\le s\le t}   X^{j}_s\ge \frac{r}{\sqrt{d}}\biggr)|\partial D_r|\,\ud r\\
	&\le 2^{d-1}|\partial D|\int_0^\infty \P\biggl(\sup_{0\le s\le t}   X^{j}_s\ge \frac{r}{\sqrt{d}}\biggr)\,\ud r
	=\sqrt{d}\hspace{1pt} 2^{d-1}|\partial D|\mu_{j,t},
\end{align*}
where $D_r$ is defined as in \eqref{def:Da}.
In terms of the second integral, note that $\delta_{D}(x)> a$ for all $x\in D_a$, and hence, for any small $t$ satisfying $\mu_{j,t}<a/\sqrt{d}$ (so that $\delta_{D}(x)/\sqrt{d}> a/\sqrt{d}>\mu_{j,t}$), 
the Borell--TIS inequality \eqref{Borell-TIS} yields
\begin{align*}
	I_{2,j}(t)
	&:=\int_{D_a} \P\biggl(\sup_{0\le s\le t}   X^{j}_s\ge \frac{\delta_{D}(x)}{\sqrt{d}}\biggr)\,\ud x
	\le 2\int_{D_a} \exp\left(-\frac{(\frac{\delta_{D}(x)}{\sqrt{d}}-\mu_{j,t})^2}{2\sigma_{j,t}^2}\right)\,\ud x\\
	&\le 2|D_a| \exp\left(-\frac{(\frac{a}{\sqrt{d}}-\mu_{j,t})^2}{2\sigma_{j,t}^2}\right).
\end{align*}
Combining the above estimates yields \eqref{spectral.upper1} for any small $t$ satisfying $\max_{1\le j\le d}\mu_{j,t}<a/\sqrt{d}$. 

Moreover, by the trivial inequality $e^{-x}< x^{-1}$ valid for $x>0$ and the inequality $0<\mu_{j,t}/(d\cdot \overline{\mu}_t) \le 1$, it follows that 
	\[
		\frac{1}{\overline{\mu}_t}\sum_{j=1}^d \exp\left(-\frac{(\frac{a}{\sqrt{d}}-\mu_{j,t})^2}{2\sigma_{j,t}^2}\right)
		\le \sum_{j=1}^d \frac{2}{(\frac{a}{\sqrt{d}}-\mu_{j,t})^2} \cdot \frac{\sigma_{j,t}^2}{\mu_{j,t}}\cdot \frac{\mu_{j,t}}{\overline{\mu}_t}
		\le \sum_{j=1}^d \frac{2d}{(\frac{a}{\sqrt{d}}-\mu_{j,t})^2} \cdot \frac{\sigma_{j,t}^2}{\mu_{j,t}},
	\]
which approaches 0 as $t\downarrow 0$ due to the observation in \eqref{sigma-mu-relation} applied to each component. This, together with  \eqref{spectral.upper1}, yields \eqref{spectral.upper2}. 
\end{proof}

In the special case where the components of the process $X$ are i.i.d.\ with common expected running supremum $\mu_t$, Proposition \ref{prop:GaussianMarkov_spectral_upperboundonly} yields the upper bound  
$\limsup_{t\downarrow 0} (|D|-Q^X_D(t))/\mu_t
	\le d^{3/2} 2^{d} |\partial D|.$  
As for the lower bound, combining Theorem \ref{thm:GaussianMarkov_regular_high} with the inequality in \eqref{Q-and-H} yields
$
	|\partial D|/\sqrt{2\pi}
 	\le \liminf_{t\downarrow 0} (|D|-Q^X_D(t))/\sqrt{R_t},
$
where $R_t$ is the common variance function;
however, this is not an ideal expression since the one-dimensional case in Theorem \ref{thm:GaussianMarkov_spectral} indicates that an appropriate rate function should be the common expected running supremum $\mu_t$.
In contrast, Theorem \ref{thm:GaussianMarkov_spectral_high} below establishes the exact asymptotic behavior with rate function $\mu_t$ by imposing some additional assumptions.

In the remainder of the paper, let $\D([0,1];\R)$ denote the space of all c\`adl\`ag functions on $[0,1]$ taking values in $\R$. The space is equipped with the standard Skorokhod $J_1$ topology, under which $x_n=(x_n(t))_{t\in [0,1]}\to x=(x(t))_{t\in [0,1]}$ as $n\to \infty$ if and only if there exists a sequence of 
strictly increasing bijections $\lambda_n:[0,1]\to [0,1]$
such that 
\begin{align}\label{def:J1}
	\sup_{t\in[0,1]}| \lambda_n(t)-t| \to 0 \ \ \ \textrm{and} \ \ \ \sup_{t\in[0,1]} |x_n(\lambda_n(t))-x(t)| \to 0 
\end{align}
as $n\to\infty$. Convergence in the $J_1$ topology coincides with convergence in the uniform topology if $x_n$ and $x$ belong to the space $C([0,1];\R)$ of continuous functions on $[0,1]$. For details about the $J_1$ topology, see e.g., \cite{Whitt}.
Unlike Proposition \ref{prop:GaussianMarkov_spectral_upperboundonly}, which gives an upper bound only for the asymptotic behavior of $|D|-Q^X_D(t)$ in a general setting, the next result establishes the exact asymptotic behavior in a restricted setting with an assumption concerning weak convergence.

\begin{theorem}\label{thm:GaussianMarkov_spectral_high}
Let $D$ be a non-empty, bounded, connected $C^{1,1}$ open set in $\R^d$  with $d\ge 2$. 
Let $X=(X_t)_{t\in[0,1]}=(X_t^1,\ldots,X_t^d)_{t\in [0,1]}$ be a stochastic process in $\R^d$ whose components are i.i.d.\ zero-mean Gaussian processes with c\`adl\`ag paths starting at 0 and common expected running supremum $\mu_t$ under the probability measure $\P=\P_0$ 
such that the almost sure boundedness condition \eqref{a.s.boundedness} holds and $X_t$ is non-degenerate for all $t\in (0,1]$.
For each $t\in (0,1]$, define a scaled process ${Y}^{(t)}=({Y}^{(t)}_{u})_{u\in [0,1]}$ of $X^{1}$  by
${Y}_{u}^{(t)}=\mu_t^{-1}X_{tu}^{1}.$
If $Y^{(t)}$ converges weakly  as $t\downarrow 0$ to a continuous Gaussian process $Y=(Y_s)_{s\in[0,1]}$ in the space $\D([0,1];\R)$ equipped with the Skorokhod $J_1$ topology, then 
$
	\E[\sup_{s\in[0,1]}Y_s]\le 1,
$ 
and
\begin{align}\label{spectral.lower1}
	 \lim_{t\downarrow 0} \frac{|D|-Q^X_D(t)}{\mu_t}
	= |\partial D|\cdot  \E[\sup_{s\in[0,1]} Y_s].
\end{align}
\end{theorem}

The proof of Theorem \ref{thm:GaussianMarkov_spectral_high}, which requires a series of lemmas and technical arguments, is postponed to Section \ref{proof}.

\begin{remark}
\begin{em}
 If the Gaussian process $X$ in Theorem \ref{thm:GaussianMarkov_spectral_high}  is continuous and self-similar with index $H>0$, then 
$Y^{(t)}=^\mathrm{d}Y^{(1)}$ for all $t\in(0,1]$, and in particular, the weak convergence of $Y^{(t)}$ to $Y:=Y^{(1)}$ trivially follows. 
Moreover, the expectation appearing on the right-hand side of \eqref{spectral.lower1} simplifies to 1; i.e.,
 $
 	\E[\sup_{s\in[0,1]}Y_s]=\E[\sup_{s\in[0,1]}Y_s^{(1)}]=1. 
 $
As a result, \eqref{spectral.lower1} takes the form
$
	 \lim_{t\downarrow 0} t^{-H}(|D|-Q^X_D(t))
	= |\partial D|\cdot \mu_1,
$
 which coincides with \eqref{selfsimilar-shc} when $d=1$. 
An example of a non-self-similar Gaussian process to which Theorem \ref{thm:GaussianMarkov_spectral_high} is applicable is provided in Example \ref{example:time-changedBM} in Section \ref{section_examples}.
\end{em}
\end{remark}

\section{Applications}\label{section_examples}

The theorems stated in the previous sections are applicable to a wide class of Gaussian processes and immediately yield a number of new statements, some of which are stated below. This section is divided into two subsections; the first subsection focuses on the one-dimensional case with the domain $D$ being a non-empty open interval, and the second provides non-self-similar Gaussian processes in $\R^d$ with $d\ge 2$ to which Theorem \ref{thm:GaussianMarkov_spectral_high} is applicable.

\subsection{Examples of Gaussian processes in $\R$}

Throughout this subsection, $D$ is a non-empty, bounded interval in $\R$.

\begin{example}
\begin{em}
 Let $X$ be a bi-fractional Brownian motion $B^{H,K}=(B^{H,K}_t)_{t\in [0,1]}$ with parameters $H\in(0,1)$ and $K\in(0,1]$, which has covariance function 
$
	R^{H,K}(s,t)=2^{-K}((s^{2H}+t^{2H})^K-|s-t|^{2HK}). 
$
Note that the special case when $K=1$ yields a fractional Brownian motion with Hurst index $H\in(0,1)$.
The bi-fractional Brownian motion has the variance function $R_t=t^{2HK}$ and is self-similar with index $HK$ (see e.g., \cite{Houdre2003_bifbm}), so  
\[
	\mu_t=t^{HK}\mu_1(H,K), \ \ \textrm{where} \ \ \mu_1(H,K):=\E\biggl[\sup_{s\in[0,1]}B^{H,K}_s\biggr]. 
\] 
(Note that even when $K=1$, the value of $\mu_1(H,K)$ is unknown except when $H=1/2$.)
Thus, by Theorems \ref{thm:GaussianMarkov_regular_high} and \ref{thm:GaussianMarkov_spectral}, 
\[
	\lim_{t\downarrow 0}\frac{|D|-H_{D}^{B^{H,K}}(t)}{t^{HK}} =\sqrt{\frac{2}{\pi}}
	\ \ \ \textrm{and} \ \ \  
	\lim_{t\downarrow 0}\frac{|D|-Q_D^{B^{H,K}}(t)}{t^{HK}} =2\mu_1(H,K).
\]
\end{em}
\end{example}

\begin{example}
\begin{em}
 Let $X$ be a fractional Ornstein--Uhlenbeck process defined by 
\[
	U^H_t:=e^{-at} \int_0^t e^{as}\ud B^H_s
\]
with $a>0$, where the integral is understood as the pathwise Riemann--Stieltjes integral with respect to a factional Brownian motion $B^H$ with Hurst index $H> 1/2$, as discussed in \cite{Cheridito2003}. By \cite[Equation (3.1)]{Mishura2020_fOU}, which builds upon \cite{Cheridito2003}, 
the variance function of $U^H$ is given by 
\[
	R_t
	=\E[(U^H_t)^2]
	=H(2H-1)e^{-2at} \iint_{[0,t]^2} e^{a(u+v)}|u-v|^{2H-2}\ud u\ud v. 
\]
By the change of variables $u=tx$ and $v=ty$,
\[
	R_t=2H(2H-1)t^{2H} \iint_{0\le x<y\le 1} e^{a(x+y-2)t}(y-x)^{2H-2}\ud x \ud y.
\]
Since $x+y-2<0$ whenever $0\le x<y\le 1$, the integrand increases as $t\downarrow 0$, so by the monotone convergence theorem, 
\[
	\lim_{t\downarrow 0} \frac{R_t}{t^{2H}}=2H(2H-1)\iint_{0\le x<y\le 1} (y-x)^{2H-2}\ud x \ud y=1.
\]
Hence, $R_t\sim t^{2H}$ as $t\downarrow 0$.
By Theorems \ref{thm:GaussianMarkov_regular_high} and \ref{thm:GaussianMarkov_spectral}, 
\[
	\lim_{t\downarrow 0}\frac{|D|-H_{D}^{U^H}(t)}{t^{H}} =\sqrt{\frac{2}{\pi}}
	\ \ \ \textrm{and} \ \ \ 
	\lim_{t\downarrow 0}\frac{|D|-Q_D^{U^H}(t)}{\mu_t} =2, 
\]
where the asymptotic behavior of $\mu_t$ is unknown as far as the authors know. 
Note that a similar argument can be made for a standard (non-fractional) Ornstein--Uhlenbeck process, in which case, the variance function satisfies
$R_t=(1-e^{-2at})/(2a) \sim t$ as $t\downarrow 0.$
\end{em}
\end{example}

\begin{example}
\begin{em}
 Following \cite{Mishura2017_ExpectedMaxima}, consider a continuous zero-mean Gaussian process $X$ for which there exist some constants $C_1,C_2>0$ and $H_1,H_2\in(0,1)$ such that
\begin{align}\label{quasi-helix0}
		C_1|s-r|^{H_1}\le \bigl(\E[|X_s-X_r|^2]\bigr)^{1/2} \le C_2|s-r|^{H_2}
\end{align}
for all $r,s\in [0,1]$. The process is called \textit{quasi-helix} if $H_1=H_2$. Examples of quasi-helix processes with $C_1< C_2$ include bi-fractional Brownian motions and sub-fractional Brownian motions. Condition \eqref{quasi-helix0} implies that 
\begin{align}\label{quasi-helix1}
	C_1 t^{H_1} \le \sqrt{R_t}\le C_2 t^{H_2}.
\end{align}
 For each fixed $t\in (0,1]$, define a continuous zero-mean Gaussian process $V^{(t)}=(V^{(t)}_s)_{s\in [0,1]}$ by letting $V_s^{(t)}:=X_{st}$. Then $\mu_t=\E[\sup_{r\in[0,t]}X_r]=\E[\sup_{s\in[0,1]}V^{(t)}_s]$ and 
\[
		C^{(t)}_1 |s-r|^{H_1}\le 
		\bigl(\E[|V^{(t)}_s-V^{(t)}_r|^2]\bigr)^{1/2}
		\le C^{(t)}_2 |s-r|^{H_2}
\]
for all $s,r\in[0,1]$, where $C^{(t)}_i=C_it^{H_i}$, $i=1,2$. Therefore, application of \cite[Theorem 1]{Mishura2017_ExpectedMaxima} for each fixed $t\in (0,1]$ yields
$
	C^{(t)}_1/(5\sqrt{H_1})\le \E[\sup_{s\in[0,1]}V^{(t)}_s] \le 16.3 C^{(t)}_2/\sqrt{H_2}.
$
In other words, 
\begin{align}\label{quasi-helix2}
	\frac{C_1}{5\sqrt{H_1}}t^{H_1}\le \mu_t \le \frac{16.3 C_2}{\sqrt{H_2}}t^{H_2}.
\end{align}
Theorems \ref{thm:GaussianMarkov_regular_high} and \ref{thm:GaussianMarkov_spectral},
together with \eqref{quasi-helix1} and \eqref{quasi-helix2}, now give information about the asymptotic behaviors of $H_D^{X}(t)$ and $Q_D^{X}(t)$, respectively. 
\end{em}
\end{example}

\begin{example}\label{example:time-changedBM-1dimension}
\begin{em}
Let a stochastic process $X$ be a Brownian motion $B$ (with $\E[B_t^2]=t$) time-changed by a deterministic, continuous, non-decreasing function $\alpha:[0,1]\to [0,\infty)$ with $\alpha(0)=0$ and $\alpha(t)>0$ for $t\in(0,1]$; i.e., $X_t=B_{\alpha(t)}$ for $t\in[0,1]$. 
For example, $X$ can be taken to be an It\^o integral of a deterministic integrand due to L\'evy's characterization theorem: 
\[
	X_t=\int_0^t f(s)\ud \tilde{B}_s =^{\textrm{d}} B_{\alpha(t)}, \ \ \ \textrm{where} \ \ \alpha(t)=[X]_t = \int_0^t f^2(s)\ud s.
\]
The process $X$ is a continuous Gaussian process that is not necessarily self-similar.
For $t\in(0,1]$, observe that 
$
	R_t = \E[(B_{\alpha(t)})^2] = \alpha(t),
$ 
while 
since $\alpha$ is continuous, non-decreasing, and starting at 0, 
\begin{align}\label{time-changedBM-mu_t}
	\mu_t = \E[\sup_{s\in[0,t]}B_{\alpha(s)}] = \E[\sup_{u\in[0,1]}B_{\alpha(t)u}]
	=\sqrt{\alpha(t)}\E[\sup_{u\in[0,1]}B_{u}]
	=\sqrt{\frac{2\alpha(t)}{\pi}},
\end{align}
where the last equality follows from e.g., \cite[Problem 2.8.2]{KaratzasShreve}.
By Theorems \ref{thm:GaussianMarkov_regular_high} and \ref{thm:GaussianMarkov_spectral}, 
\[
	\lim_{t\downarrow 0}\frac{|D|-H_{D}^{B\circ \alpha}(t)}{\sqrt{\alpha(t)}} =\sqrt{\frac{2}{\pi}}
	\ \  \ \textrm{and} \ \ \ 
	\lim_{t\downarrow 0}\frac{|D|-Q_D^{B\circ \alpha}(t)}{\sqrt{\alpha(t)}} =2\sqrt{\frac{2}{\pi}}. 
\]
\end{em}
\end{example}

\subsection{Examples of Gaussian processes in $\R^d$ with $d\ge 2$}

This subsection provides non-self-similar Gaussian processes in $\R^d$ with $d\ge 2$ to which Theorem \ref{thm:GaussianMarkov_spectral_high} is applicable.
We first derive a lemma that gives sufficient conditions for the weak convergence of $Y^{(t)}$ to a Gaussian process $Y$, as stated in Theorem \ref{thm:GaussianMarkov_spectral_high}.

\begin{lemma}\label{lemma:weakconvergence0}
Let $X=(X_t)_{t\in[0,1]}=(X_t^1,\ldots,X_t^d)_{t\in [0,1]}$ be a stochastic process in $\R^d$ as in Theorem \ref{thm:GaussianMarkov_spectral_high}.
Assume that there exists a continuous function $g$ on $[0,1]^2$ such that for each $s,r\in[0,1]$, 
\begin{align}\label{scaledcovariance}
	\frac{R_{X^{1}}(ts,tr)}{\mu_t^2}\to g(s,r) \  \ \textrm{as} \ \ t\downarrow 0,
\end{align}
 where $R_{X^{1}}(s,r)=\E[X^{1}_sX^{1}_r]$. Assume further that there exist a constant $\kappa>1$ and a continuous, nondecreasing function $h$ on $[0,1]$ such that for any $t\in(0,1]$ and $0\le r<s<u\le 1$, the scaled moment condition 
\begin{align}\label{momentcondition}
	\frac{1}{\mu_t^4}\E\Bigl[(X^{1}_{tu}-X^{1}_{ts})^2(X^{1}_{ts}-X^{1}_{tr})^2\Bigr]\le (h(u)-h(r))^\kappa
\end{align}
holds. 
Then as $t\downarrow 0$, $Y^{(t)}$ defined in Theorem \ref{thm:GaussianMarkov_spectral_high} converges weakly to a Gaussian process $Y$ with covariance function $R_Y(s,r)=g(s,r)$ in the space $\D([0,1];\R)$ equipped with the Skorokhod $J_1$ topology.
\end{lemma}

\begin{proof}
Observe that $Y^{(t)}$ is a zero-mean Gaussian process in $\R$ with covariance function 
$
	R_{Y^{(t)}}(s,r)= \mu_t^{-2}R_{X^{1}}(ts,tr),
$
 which converges to $g(s,r)$ as $t\downarrow 0$ due to assumption \eqref{scaledcovariance}. 
In particular, since the function $(s,r)\mapsto R_{Y^{(t)}}(s,r)$ is nonnegative definite for each fixed $t$, so is the limiting function $(s,r)\mapsto g(s,r)$; thus, the latter is the covariance function of some Gaussian process $Y=(Y_s)_{s\in[0,1]}$ in $\R$. Now, the convergence of the covariance function as $t\downarrow 0$ implies the weak convergence of any linear combination of the form $\sum_{k=1}^n a_k Y^{(t)}_{s_k}$ to $\sum_{k=1}^n a_k Y_{s_k}$, where $0\le s_1<s_2<\cdots<s_n\le 1$ and $a_1,\ldots,a_n\in\R$. Therefore, by the Cram\'er--Wold device (see e.g., \cite[Theorem 3.10.6]{Durrett-5th}), it follows that 
$Y^{(t)}\to^{\textrm{f.d.d.}} Y$ as $t\downarrow 0,$
where the symbol $\to^{\textrm{f.d.d.}}$ means the convergence of finite-dimensional distributions. 
Thus, it now suffices to verify the tightness of the family $\{Y^{(t)}: t\in(0,1]\}$ to obtain the desired weak convergence in the Skorokhod space. 

Note that since the covariance function $g(s,r)$ is continuous, due to \cite[Theorem 8.12]{Jansen}, the Gaussian process $Y$ is continuous in $L^2$,
and in particular,
$Y_1-Y_{1-\delta} \to^{\mathrm{d}} 0$ as $\delta\downarrow 0.$
On the other hand, the scaled moment condition \eqref{momentcondition} implies that 
$
	\E\bigl[(Y^{(t)}_{u}-Y^{(t)}_{s})^2(Y^{(t)}_{s}-Y^{(t)}_{r})^2\bigr]\le (h(u)-h(r))^\kappa
$
for any $t\in (0,1]$ and $0\le u<s<u\le 1$. 
Therefore, by \cite[Theorem 13.5]{Billingsley_CoPM_2nd}, $Y^{(t)}\to^\mathrm{d} Y$ in the space $\D([0,1];\R)$ equipped with the Skorokhod $J_1$ topology.  
\end{proof}

\begin{remark}
\begin{em}
In Lemma \ref{lemma:weakconvergence0}, assumption \eqref{momentcondition} can be replaced by any condition guaranteeing the tightness of the family $\{Y^{(t)}: t\in (0,1]\}$. For example, if $X$ has continuous paths (and hence, weak convergence is discussed on the $C$ space instead), then \eqref{momentcondition} can be replaced by the simpler moment condition
$
	\mu_t^{-2}\E\bigl[(X^{1}_{ts}-X^{1}_{tr})^2\bigr]\le (h(s)-h(r))^\kappa;
$
see \cite[Theorem 11.6.5]{Whitt}.
\end{em}
\end{remark}

We are now ready to give an example of a class of non-self-similar Gaussian processes that satisfy assumptions  \eqref{scaledcovariance} and \eqref{momentcondition} in Lemma \ref{lemma:weakconvergence0} (so the weak convergence of $Y^{(t)}$ occurs, and consequently, Theorem \ref{thm:GaussianMarkov_spectral_high} is applicable.).

\begin{example}\label{example:time-changedBM}
\begin{em}
Let $X=(X^1,\ldots,X^d)$ be a stochastic process in $\R^d$ with i.i.d.\ components that are equal in distribution to the time-changed Brownian motion discussed in Example \ref{example:time-changedBM-1dimension}; i.e., $(X^{1}_t)_{t\in[0,1]}=^{\textrm{d}}(B_{\alpha(t)})_{t\in [0,1]}$. 
The process $X^1$ is a continuous Gaussian process that is not necessarily self-similar.
Assume that the deterministic time change $\alpha(t)$ is regularly varying at 0 with index $\rho>0$, which satisfies $\alpha(ts)/\alpha(t)\to s^\rho$ as $t\downarrow 0$ for any fixed $s\in [0,1]$ (see \cite[Chapter 1]{Bingham_book}). Then for $s, r\in[0,1]$, by \eqref{time-changedBM-mu_t}, 
\[
	\frac{R_{X^{1}}(ts,tr) }{\mu_t^2}
	=\frac{\pi\E[B_{\alpha(ts)}B_{\alpha(tr)}]}{2\alpha(t)}
	=\frac{\pi (\alpha(ts) \wedge \alpha(tr))}{2\alpha(t)} \to \frac{\pi(s\wedge r)^\rho}{2}
\]
as $t\downarrow 0$,
so \eqref{scaledcovariance} holds with $g(s,r)=(\pi/2)(s\wedge r)^\rho$. 

On the other hand, for $t\in(0,1]$ and $0\le r<s<u\le 1$, since $\alpha$ is non-decreasing, it follows that 
$
	\E\bigl[(X^{1}_{tu}-X^{1}_{ts})^2(X^{1}_{ts}-X^{1}_{tr})^2\bigr] 
	=\E\bigl[(B_{\alpha(tu)}-B_{\alpha(ts)})^2(B_{\alpha(ts)}-B_{\alpha(tr)})^2\bigr] 
	=(\alpha(tu)-\alpha(ts))(\alpha(ts)-\alpha(tr)) 
	\le (\alpha(tu)-\alpha(tr))^2.
$
Hence, by \eqref{time-changedBM-mu_t},
\[
		\frac{1}{\mu_t^4}\E\Bigl[(X^{1}_{tu}-X^{1}_{ts})^2(X^{1}_{ts}-X^{1}_{tr})^2\Bigr] 
		\le \frac{\pi^2}{4}\Bigl(\frac{\alpha(tu)-\alpha(tr)}{\alpha(t)}\Bigr)^2.
\]
(Or, since this is a continuous process, we could instead discuss the quantity 
$
\E\bigl[(X^{1}_{tu}-X^{1}_{tr})^4\bigr] 
=\E\bigl[(B_{\alpha(tu)}-B_{\alpha(tr)})^4\bigr] 
=3(\alpha(tu)-\alpha(tr))^2.
$)
Although the latter approaches $(\pi^2/4)(u^\rho-r^\rho)^2$,
we rather need a uniform bound valid for all $0\le r<u\le 1$ and for all small $t$. To achieve this, 
suppose additionally that $\alpha(t)$ has representation 
 $
 	\alpha(t)=t^\rho \ell(t)
 $
in terms of a slowly varying function $\ell(t)$ at 0 satisfying the following condition: there exist constants $0<m<M<\infty$ and a continuous, non-decreasing function $H$ such that for all $t\in(0,1]$,
 \[
 	m\le \ell(t)\le M \ \ \textrm{and} \ \ 
 	\ell(tu)-\ell(tr)\le H(u)-H(r) \ \ \textrm{for all} \ \  0\le r<u\le 1.
\] 
(The second condition is satisfied if $\ell$ has a derivative with positive upper bound on $(0,1]$ due to the mean value theorem.)
Then for $0\le r<u\le 1$, 
\[
	0\le \frac{\alpha(tu)-\alpha(tr)}{\alpha(t)}
	=\frac{u^\rho(\ell(tu)-\ell(tr))+\ell(tr)(u^\rho-r^\rho)}{\ell(t)}
	\le \frac{h(u)-h(r)}{m},
\]
where $h(u):=H(u)+Mu^\rho$. 
As a result, 
 \[
		\frac{1}{\mu_t^4}\E\Bigl[(X^{1}_{tu}-X^{1}_{ts})^2(X^{1}_{ts}-X^{1}_{tr})^2\Bigr] 
		\le \frac{\pi^2}{4m^2}((h(u)-h(r))^2,
\]
so \eqref{momentcondition} is satisfied with $\kappa=2$. 

Now, by Lemma \ref{lemma:weakconvergence0}, $Y^{(t)}$ defined in Theorem \ref{thm:GaussianMarkov_spectral_high} converges weakly to a Gaussian process $Y$ with covariance function $R_Y(s,r)=(\pi/2)(s\wedge r)^\rho$ in the space $\D([0,1];\R)$.
Furthermore, for any $0\le r<s\le 1$,  
$\E[|Y_s-Y_r|^2]=(\pi/2)(s^\rho-r^\rho) \le (\pi/2)\rho(s-r)$ if $\rho>1$, while  $\E[|Y_s-Y_r|^2]\le (\pi/2)(s-r)^\rho$ if $0<\rho\le 1$,
where we used the mean value theorem when $\rho>1$ and the sub-additivity of $x\mapsto x^\rho$ when $0<\rho\le 1$. In either case, for some $\delta>0$, the latter upper bound is dominated above by $(\log(s-r))^{-2}$ as long as $0<s-r<\delta$. Hence, criterion (1.4.4) in \cite[Theorem 1.4.1]{AdlerTaylor_book} is satisfied, and consequently, $Y$ is a continuous Gaussian process, as stated in Theorem \ref{thm:GaussianMarkov_spectral_high}. Thus, we can now apply Theorem \ref{thm:GaussianMarkov_spectral_high} to obtain the asymptotic statement in \eqref{spectral.lower1}.

Examples of time changes $\alpha(t)$ that satisfy the conditions assumed above include $\alpha(t)=t^\rho\log(t+c)$ with $c>1$, and
$\alpha(t)=\sum_{k=0}^n c_k t^{\eta_k}$ with $c_0> 0$, $c_k\ge 0$ for all $k\ge 1$, and $0< \eta_0<\eta_1<\eta_2<\cdots<\eta_n$. In fact, for the latter time change, one can take $\rho=\eta_0$, $m=c_0$, $M=\sum_{k=0}^n c_k$, and $H(u)=\sum_{k=0}^n c_k u^{\eta_k-\eta_0}$.
\end{em}
\end{example}

\section{Proof of Theorem \ref{thm:GaussianMarkov_spectral_high}}\label{proof}

Throughout this section, we assume that $X$, $Y$, $Y^{(t)}$, and $D$ are those appearing in the statement of Theorem \ref{thm:GaussianMarkov_spectral_high}. 
In particular, recall that $X^1=(X^1_t)_{t\in[0,1]}$ denotes the first component of the process $X=(X^1,\ldots,X^d)$ and that ${Y}_{u}^{(t)}=\mu_t^{-1}X_{tu}^{1}.$
The first lemma guarantees that the right-hand side of \eqref{spectral.lower1} is finite. 

\begin{lemma}\label{lemma:finiteness}
$\E[\sup_{s\in[0,1]}Y_s]\le 1$. 
\end{lemma}

\begin{proof}
For each fixed $t>0$, observe that
$
	\E[\sup_{s\in[0,1]}Y_s^{(t)}]=\mu_t^{-1}\E[\sup_{s\in[0,1]}X_{ts}^{1}]=1.
$
For fixed $\eps>0$ and $n\ge 1$, the weak convergence $(Y^{(t)}_s)_{s\in[0,1]}\to^\mathrm{d} (Y_s)_{s\in[0,1]}$ in the space $\D([0,1];\R)$ implies that $\E[\sup_{s\in[0,1]} Y_s^{(t)}\wedge n]\to \E[\sup_{s\in[0,1]} Y_s\wedge n]$ as $t\downarrow 0$ (as the supremum and minimum are both continuous functions on the Skorokhod space), so there exists $t_0(n)=t_0(n,\eps)>0$ such that 
$
	\bigl|\E[\sup_{s\in[0,1]} Y_s^{(t_0(n))}\wedge n]- \E[\sup_{s\in[0,1]} Y_s\wedge n]\bigr|<\eps.
$
By Fatou's lemma, 
\begin{align*}
	\E[\sup_{s\in[0,1]} Y_s]
	&\le \liminf_{n\to\infty} \E[\sup_{s\in[0,1]} Y_s\wedge n]
	\le \liminf_{n\to\infty} \left(\E[\sup_{s\in[0,1]} Y_s^{(t_0(n))}\wedge n]+\eps\right)\\
	&\le \liminf_{n\to\infty} \left(\E[\sup_{s\in[0,1]} Y_s^{(t_0(n))}]+\eps\right)
	= 1+\eps.
\end{align*}
Letting $\eps\downarrow 0$ completes the proof. 
\end{proof}

The next lemma shows that the weak convergence of $Y^{(t)}$ in Theorem \ref{thm:GaussianMarkov_spectral_high} implies the weak convergence of a scaled version of the process $X$ in $\R^d$. Here, the reader should be alerted that  two kinds of the $J_1$ topology can be considered in the multi-dimensional setting. 
One is the product topology on the Cartesian product $\prod_{j=1}^d \D([0,1];\R):=\D([0,1];\R)\times \cdots \times \D([0,1];\R)$,
with which the convergence $x_n=(x^1_n,\ldots,x^d_n)\to x=(x^1,\ldots,x^d)$ is equivalent to the convergence $x^j_n \to x^j$ in $\D([0,1];\R)$ for each $1\le j\le d$. The other is the stronger (finer) topology that is obtained by using the maximum norm of $\R^d$, $\|a\|:=\max_{1\le j\le d} |a^j|$, in \eqref{def:J1}. Below, the $J_1$ topology for the space $\D([0,1];\R^d)$ means this stronger topology. In addition, $C([0,1];\R^d)$ denotes the space of $\R^d$-valued continuous functions on $[0,1]$.

\begin{lemma}\label{lemma:weakconvergence}
For each $t\in (0,1]$, define a scaled process $\widetilde{X}^{(t)}=(\widetilde{X}^{(t)}_{u})_{u\in [0,1]}$ of $X$  by
$\widetilde{X}_{u}^{(t)}=\mu_t^{-1}X_{tu}$. 
Then as $t\downarrow 0$, $\widetilde{X}^{(t)}$ converges weakly to $Z$ in $\D([0,1];\R^d)$ equipped with the Skorokhod $J_1$ topology, where $Z$ is a process whose components are independent copies of the Gaussian process $Y$. 
\end{lemma}

\begin{proof}
Express $\widetilde{X}^{(t)}$ as
$
	\widetilde{X}^{(t)}=(\widetilde{X}^{(t)1},\ldots, \widetilde{X}^{(t)d}), 
$  
where $\widetilde{X}^{(t)j}$'s are i.i.d.\ and have the same distribution as $Y^{(t)}$.  By assumption, for each fixed $j$, $\widetilde{X}^{(t)j}$ converges weakly to a continuous process $Z^{j}$ in $\D([0,1];\R)$ with the $J_1$ topology, where $Z^j$'s have the same distribution as $Y$. Since $\widetilde{X}^{(t)j}$'s are independent, so are $Z^j$'s, and one can consider the process
$
	Z=(Z^{1},\ldots, Z^{d})
$
on a single probability space. Since $\widetilde{X}^{(t)j}$ converges weakly to $Z^{j}$ in $\D([0,1];\R)$ with the $J_1$ topology for each $j$, $\widetilde{X}^{(t)}$ converges weakly to $Z$ in $\prod_{j=1}^d \D([0,1];\R)$ equipped with the product topology.
However, $\widetilde{X}^{(t)}$ in fact converges weakly to $Z$ in $\D([0,1];\R^d)$ with the stronger topology. 
To verify this, consider the identity map $\textrm{id}: \prod_{j=1}^d \D([0,1];\R) \to \D([0,1];\R^d)$. 
Since the convergence $x_n=(x^1_n,\ldots,x^d_n)\to x=(x^1,\ldots,x^d)$ in $\prod_{j=1}^d \D([0,1];\R)$ does not imply the convergence $x_n\to x$ in $\D([0,1];\R^d)$, the identity map \textrm{id} is not continuous on the entire space $\prod_{j=1}^d \D([0,1];\R)$; for details, see \cite[Sections 3 and 12]{Whitt}. 
However, the $J_1$ version of \cite[Theorem 12.6.1]{Whitt} implies that $\textrm{id}$ is continuous at any $x=(x^1,\ldots,x^d)\in C([0,1];\R^d)$; in other words, $C([0,1];\R^d)\subset (\textrm{Disc}(\textrm{id}))^c$, 
where $\textrm{Disc}(\textrm{id})$ denotes the set of all discontinuities of $\textrm{id}$.
Since $Z$ has continuous paths, $\P(Z\in \textrm{Disc}(\textrm{id}))\le \P(Z\in (C([0,1];\R^d))^c)=0$. Therefore, by the continuous mapping theorem (see e.g., \cite[Theorem 3.4.3]{Whitt}), the weak convergence $\widetilde{X}^{(t)}\to^{\mathrm{d}} Z$ in $\prod_{j=1}^d \D([0,1];\R)$ yields the weak convergence $\widetilde{X}^{(t)}\to^{\mathrm{d}} Z$ in $\D([0,1];\R^d)$, as desired. 
\end{proof}

With Lemmas \ref{lemma:finiteness} and \ref{lemma:weakconvergence} at hand, we will prove Theorem \ref{thm:GaussianMarkov_spectral_high} following the approach taken in \cite{KP23-unified} for general L\'evy processes with regularly varying characteristic exponents. 
The next lemma shows that the amount of heat loss from deep inside the domain $D$ is negligible when compared to $\mu_t$.
Recall the notation $D_a$ defined in \eqref{def:Da}. 

\begin{lemma}\label{lemma:inside}
Let $R$ be the characteristic radius of the $C^{1,1}$ domain $D$.
 Then for any fixed $a\in(0,R/2]$,  as $t\downarrow 0$, 
\[
\int_{D_a}\P_{x}(\tau_{D}^{X}\leq t)\ud x =o(\mu_t).
\]
\end{lemma}

\begin{proof}
Let $t\in (0,1]$ be small so that $a>\mu_t$. For any $x\in D_a$,
\begin{align*}
\P_{x}( \tau_{D}^X \leq t) 
\leq \P_x\left(\sup _{s\in[0,t]}|X_{s}-x| \geq a\right)
=\P\left(\sup _{s\in[0,t]}|X_{s}| \geq a\right)
=2\P\left(\sup _{s\in[0,t]}X_{s} \geq a\right)
\le 4e^{-\frac{(a-\mu_t)^2}{2\sigma_t^2}}, 
\end{align*}
where we used the Borell--TIS inequality \eqref{Borell-TIS}. The inequality $e^{-x}<x^{-1}$ valid for $x>0$ together with the asymptotic relation of $\sigma^2_t/\mu_t$ obtained from \eqref{sigma-mu-relation} yields the desired statement. 
\end{proof}

Now we start investigating the quantity $\int_{D\setminus D_a}\P_{x}(\tau_{D}^{X}\leq t)\ud x$, which is the amount of heat loss from points near the boundary of the domain $D$. Throughout, the notation
\[
	H=\{x=(\tilde{x}, x^{d})=(x^1,\ldots,x^{d-1},x^d)\in \R^d : x^{d}>0\}
\]
 denotes the upper half-space, and $\tilde{0}$ represents the zero vector in $\R^{d-1}$; in particular, a given point in $H$ that is on the last coordinate axis can be represented as $(\tilde{0},u)$ with $u>0$.

\begin{lemma}\label{lemma:half-space}
Let $R$ be the characteristic radius of the $C^{1,1}$ domain $D$.
Then for any fixed $a\in(0, R/2]$,
\[
\lim_{t\downarrow 0}\mu_t^{-1}\int_{0}^{a}\P_{(\tilde{0},u)}(\tau^{X}_{H} \leq t)\ud u=\E[\sup_{s\in[0,1]} Y_s].
\]
\end{lemma}

\begin{proof}
Recall the family of scaled processes $\{\widetilde{X}^{(t)}:t\in (0,1]\}$ defined in Lemma \ref{lemma:weakconvergence}.  
Note that 
\begin{align*}
\tau_{H}^{X}
&=\inf\{r : X_{r}\notin H\}
=\inf\{tr: \mu_t^{-1} X_{tr}\notin \mu_t^{-1} H\}\\
&=\inf\{tr: \mu_t^{-1} X_{tr}\notin H\}
=t\inf\{r: \widetilde{X}_{r}^{(t)} \notin H\}=t\tau_{H}^{\widetilde{X}^{(t)}},
\end{align*}
where we used the fact that the upper half-space $H$ is invariant under the multiplication by any positive constant. The latter
 shows that the law of $\tau_{H}^{X}$ under $\P_{x}$ is equal to the law of $t\tau_{H}^{\widetilde{X}^{(t)}}$ under $\P_{\mu_t^{-1}x}$.
Hence, for a fixed $a\in(0,R/2]$, by the change of variables $v=\mu_t^{-1}u$,
\begin{align}\label{eqn:half-space}
\int_{0}^{a}\P_{(\tilde{0},u)}(\tau_{H}^X \leq t)\ud u
=\int_{0}^{a}\P_{(\tilde{0},\mu_t^{-1}u)}(\tau_{H}^{\widetilde{X}^{(t)}} \leq 1)\ud u
=\mu_t\int_{0}^{a \mu_t^{-1}}\P_{(\tilde{0},v)}(\tau_{H}^{\widetilde{X}^{(t)}}\leq  1)\ud v.
\end{align}
By Lemma \ref{lemma:finiteness}, for any $\eps>0$, there exists $N=N(\eps)>0$ such that 
$
\int_{0}^{N}\P(\sup_{s\in[0,1]} Y_s> v)\ud v>\E[\sup_{s\in[0,1]} Y_s]-\eps.
$
Since $\mu_t^{-1}\to \infty$ as $t\downarrow 0$, there exists $t_{0}>0$ such that $a\mu_t^{-1} \geq N$ for all $0<t\leq t_{0}$. Hence, it follows from \eqref{eqn:half-space}, Fatou's lemma, and the assumption on weak convergence of $Y^{(t)}$ to $Y$ that
\begin{align*}
&\liminf_{t\downarrow 0}\mu_t^{-1}\int_{0}^{a}\P_{(\tilde{0},u)}(\tau_{H}^X\leq t)\ud u
\geq 
\liminf_{t\downarrow 0}\int_{0}^{N}\P_{(\tilde{0},v)}(\tau_{H}^{\widetilde{X}^{(t)}}\leq 1)\ud v\\
&=\liminf_{t\downarrow 0} \int_{0}^{N}\P_v(\inf_{s\in[0,1]} Y_s^{(t)}\le 0)\ud v
=\liminf_{t\downarrow 0} \int_{0}^{N}\P(\inf_{s\in[0,1]} (Y_s^{(t)}+v)\le 0)\ud v\\
&=\liminf_{t\downarrow 0} \int_{0}^{N}\P(\sup_{s\in[0,1]} Y_s^{(t)}> v)\ud v
\ge \int_{0}^{N}\P(\sup_{s\in[0,1]} Y_s> v)\ud v 
>\E[\sup_{s\in[0,1]} Y_s]-\eps.
\end{align*}
Since $\eps>0$ is arbitrary, the lower bound  
$
\liminf_{t\downarrow 0}\mu_t^{-1}\int_{0}^{a}\P_{(\tilde{0},u)}(\tau_{H}^X\leq t)\ud u\geq \E[\sup_{s\in[0,1]} Y_s]
$
follows.

Derivation of the upper bound requires a delicate discussion. In fact, a simple modification of the above argument  would be to notice by \eqref{eqn:half-space} that 
\[
\limsup_{t\downarrow 0}\mu_t^{-1}\int_{0}^{a}\P_{(\tilde{0},u)}(\tau_{H}^X \leq t)\ud u
= \limsup_{t\downarrow 0}\int_{0}^{\infty}\P(\sup_{s\in[0,1]} Y_s^{(t)}\geq v)\mathbf{1}_{(0,a\mu_t^{-1}]}(v)\ud v
\]
and use the reverse version of Fatou's lemma for the latter limit. However, that would require the integrand $\P(\sup_{s\in[0,1]} Y_s^{(t)}\geq v)\mathbf{1}_{(0,a\mu_t^{-1}]}(v)$ to be bounded above by a function of $v$ that is both independent of $t$ and integrable on the unbounded interval $(0,\infty)$, and finding such an upper bound is a non-trivial task. 
To overcome this hurdle, fix $M>2$ and $t\in (0,1]$ so that $a\mu_t^{-1}>M$, and note by the Borell--TIS inequality \eqref{Borell-TIS} that 
\begin{align*}
	\int_{M}^{a\mu_t^{-1}}\P(\sup_{s\in[0,1]} Y_s^{(t)}> v)\ud v
	&=\int_{M}^{a\mu_t^{-1}}\P(\sup_{r\in[0,t]} X_r^{1}> v\mu_t)\ud v
	\le \int_{M}^{a\mu_t^{-1}} 2e^{-\frac{(v-1)^2\mu_t^2}{2\sigma^2_t}}\ud v\\
	&\le \int_{M}^{\infty} 2(v-1)e^{-\frac{(v-1)^2\mu_t^2}{2\sigma^2_t}}\ud v
	=\frac{2\sigma_t^2}{\mu_t^2}e^{-\frac{(M-1)^2\mu_t^2}{2\sigma_t^2}}
	\le \frac{4\sigma_t^4}{(M-1)^2\mu_t^4},
\end{align*}
where we used the inequality $e^{-x}<x^{-1}$ valid for $x>0$. 
A simple modification of the argument used to obtain \eqref{sigma-mu-relation} yields 
$\sigma_t^2\mu_t^{-2}\le 2R_{s_\ast(t)}(c_2(R_{s_\ast(t)})^{1/2})^{-2} = 2c_2^{-2}$
for all $t>0$. 
Therefore, 
\begin{align}\label{eqn:ub2}
	\int_{M}^{a\mu_t^{-1}}\P(\sup_{s\in[0,1]} Y_s^{(t)}> v)\ud v\le  \frac{16}{c_2^4(M-1)^2}.
\end{align}
Now, using \eqref{eqn:half-space} and \eqref{eqn:ub2}, applying the reverse version of Fatou's lemma on the bounded interval $(0,M]$, and using the weak convergence of $Y^{(t)}$ to $Y$, we obtain
\begin{align*}
&\limsup_{t\downarrow 0}\mu_t^{-1}\int_{0}^{a}\P_{(\tilde{0},u)}(\tau_{H}^X \leq t)\ud u
=\limsup_{t\downarrow 0}\int_{0}^{a \mu_t^{-1}}\P_{(\tilde{0},v)}(\tau_{H}^{\widetilde{X}^{(t)}}\leq  1)\ud v\\
&\le \limsup_{t\downarrow 0}\int_0^{M}\P(\sup_{s\in[0,1]} Y_s^{(t)}\ge  v)\ud v
+\limsup_{t\downarrow 0}\int_{M}^{a\mu_t^{-1}}\P(\sup_{s\in[0,1]} Y_s^{(t)}\ge v)\ud v\\
&\le\int_{0}^{M}\P(\sup_{s\in[0,1]} Y_s\ge v)\ud v +\frac{16}{c_2^4(M-1)^2}.
\end{align*}
Letting $M\uparrow \infty$ yields the upper bound
$
\limsup_{t\downarrow 0}\mu_t^{-1}\int_{0}^{a}\P_{(\tilde{0},u)}(\tau_{H}^X\leq t)\ud u\leq \E[\sup_{s\in[0,1]} Y_s].
$
\end{proof}

In the lemma below, $B((\tilde{0},R),R)$ represents the ball of radius $R$ centered at $(\tilde{0},R)\in H$.

\begin{lemma}\label{lemma:ball}
Let $R$ be the characteristic radius of the $C^{1,1}$ domain $D$.
Then for any fixed $a\in(0, R/2]$,
\[
\lim_{t\downarrow 0}\mu_t^{-1}\int_{0}^{a}\P_{(\tilde{0},u)}(\tau^{X}_{B((\tilde{0},R),R)}\leq t)\ud u=\E[\sup_{s\in[0,1]} Y_s].
\]
\end{lemma}

\begin{proof}
Since $B((\tilde{0},R),R)\subset H$, it follows that $\P_{(\tilde{0},u)}(\tau_{H}^X\leq t) \leq \P_{(\tilde{0},u)}(\tau^{X}_{B((\tilde{0},R),R)}\leq t)$, so by Lemma \ref{lemma:half-space}, the lower bound
$
\liminf_{t\downarrow 0}\mu_t^{-1}\int_{0}^{a}\P_{(\tilde{0},u)}(\tau^{X}_{B((\tilde{0},R),R)} \le t)\ud u\geq \E[\sup_{s\in[0,1]}Y_s]
$
follows. 

To derive the upper bound, recall the definition of $\widetilde{X}^{(t)}$ in Lemma \ref{lemma:weakconvergence} 
and observe that
\begin{align*}
\tau_{B((\tilde{0},R),R)}^{X}
&=\inf\{r: X_{r}\notin B((\tilde{0},R),R)\}=\inf\{tr: X_{tr}\notin B((\tilde{0},R),R)\}\\
&=\inf\{tr: \mu_t^{-1}X_{tr} \notin \mu_t^{-1}B((\tilde{0},R),R) \}
=t\inf\{r: \widetilde{X}_{r}^{(t)}\notin B((\tilde{0},\mu_t^{-1}R),\mu_t^{-1}R) \}\\
&=t\tau_{B((\tilde{0},\mu_t^{-1}R),\mu_t^{-1}R)}^{\widetilde{X}^{(t)}}.
\end{align*}
Hence, by the change of variables $v=\mu_t^{-1}u$, 
\begin{align}\label{eqn:ball lb1}
\int_{0}^{a}\P_{(\tilde{0},u)}(\tau^{X}_{B((\tilde{0},R),R)}\leq t)\ud u
&=\int_{0}^{a}\P_{(\tilde{0},\mu_t^{-1}u)}(\tau_{B((\tilde{0},\mu_t^{-1}R),\mu_t^{-1}R)}^{\widetilde{X}^{(t)}} \leq 1)\ud u\\
&=\mu_t\int_{0}^{a\mu_t^{-1}}\P_{(\tilde{0},v)}(\tau_{B((\tilde{0},\mu_t^{-1}R),\mu_t^{-1}R)}^{\widetilde{X}^{(t)}} \leq 1)\ud v.\notag
\end{align}

Note that if $v\in (0, a\mu_t^{-1}]$, then $B((\tilde{0}, v),v)\subset B((\tilde{0},\mu_t^{-1}R),\mu_t^{-1}R)$ since $a\le R/2$. 
Moreover, $\max_{1\leq j\leq d}|x^j|\leq |x|=(\sum_{j=1}^{d}(x^j)^2)^{1/2}\leq \sqrt{d}\max_{1\leq j\leq d}|x^j|$ for any $x=(x^1,\ldots, x^d)\in \R^{d}$. Therefore, for any $v\in (0, a\mu_t^{-1}]$, 
\begin{align*}
\P_{(\tilde{0},v)}(\tau_{B((\tilde{0},\mu_t^{-1}R),\mu_t^{-1}R)}^{\widetilde{X}^{(t)}} \leq 1)
&\leq \P_{(\tilde{0},v)}(\tau_{B((\tilde{0},v),v)}^{\widetilde{X}^{(t)}} \leq 1)
=\P\left(\sup_{u\in[0,1]}|\widetilde{X}_{u}^{(t)}|\geq v\right)\\
&\leq d\cdot \P\left(\sup_{u\in[0,1]}|Y_{u}^{(t)}|\geq \frac{v}{\sqrt{d}}\right)
= 2d\cdot \P\left(\sup_{u\in[0,1]}Y_{u}^{(t)}\geq \frac{v}{\sqrt{d}}\right)
\end{align*}
where $Y^{(t)}$ has the same distribution as the i.i.d.\ components of $\widetilde{X}^{(t)}$. As in the discussion given in Lemma \ref{lemma:half-space}, by the Borell--TIS inequality \eqref{Borell-TIS}, for any fixed $M>2$, 
\begin{align}\label{eqn:ball lb6}
	\limsup_{t\downarrow 0}\int_{M\sqrt{d}}^{a\mu_t^{-1}}\P_{(\tilde{0},v)}(\tau_{B((\tilde{0},\mu_t^{-1}R),\mu_t^{-1}R)}^{\widetilde{X}^{(t)}} \leq 1)\ud v
	&\le 2d\limsup_{t\downarrow 0}\int_{M\sqrt{d}}^{a\mu_t^{-1}} e^{-\frac{(\frac{v}{\sqrt{d}}-1)^2\mu_t^2}{2\sigma_t^2}}\ud v \\ 
	&\le \frac{32d\sqrt{d}}{c_2^4(M-1)^2}.\notag
\end{align}

Now, to deal with the integral $\int_{0}^{M\sqrt{d}}\P_{(\tilde{0},v)}(\tau_{B((\tilde{0},\mu_t^{-1}R),\mu_t^{-1}R)}^{\widetilde{X}^{(t)}} \leq 1)\ud v$, we construct an increasing sequence of domains $D(n)$ by
\[
D(n)=\left\{x=(\tilde{x}, x^{d})\in \R^{d-1}\times \R: |\tilde{x}|<n,\, 0<x^{d}<n,\, \cos^{-1}
\Bigl(\vec{\bf{n}}\cdot \frac{x}{|x|} \Bigr)
<\frac{\pi}{2}-\frac{1}{n} \right\},
\]
where $\vec{\bf{n}}=(\tilde{0},1)$ (see Figure \ref{figure_Dn}).
\begin{figure}
\begin{center}
\includegraphics[height=3.5cm,clip]{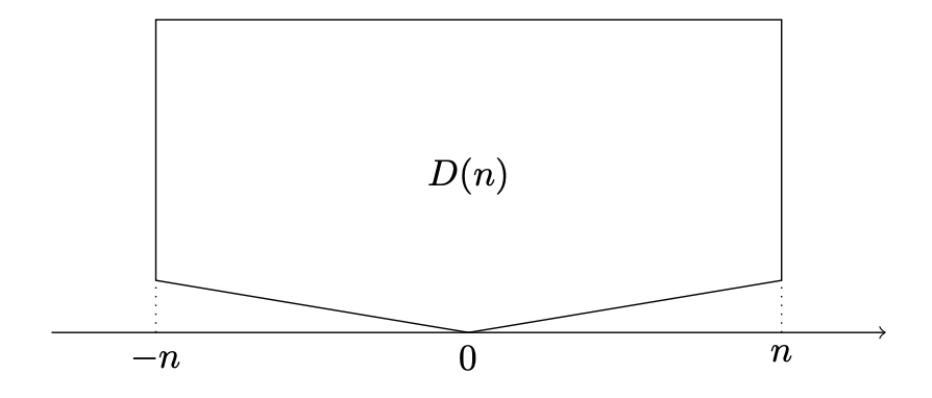}
\caption{The domain $D(n)$ increases to the upper half-space $H$.}
\label{figure_Dn}
\end{center}
\end{figure}
For the limiting process $Z$ in Lemma \ref{lemma:weakconvergence}, 
since $D(n)$ increases to $H$ as $n\to\infty$, it follows from the bounded convergence theorem that
\[
\lim_{n\to\infty}\int_{0}^{M\sqrt{d}}\P_{(\tilde{0},v)}(\tau_{D(n)}^{Z}\leq 1)\ud v=\int_{0}^{M\sqrt{d}}\P_{(\tilde{0},v)}(\tau_{H}^{Z}\leq 1)\ud v,
\]
where the sequence converges decreasingly.
Therefore, for a given $\eps>0$, we may take an integer $N=N(\eps)$ such that 
\begin{align}\label{eqn:ball lb2}
\int_{0}^{M\sqrt{d}}\P_{(\tilde{0},v)}(\tau_{D(N)}^{Z}\leq 1)\ud v<\int_{0}^{M\sqrt{d}}\P_{(\tilde{0},v)}(\tau_{H}^{Z}\leq 1)\ud v + \eps.
\end{align}
Since $B((\tilde{0},\mu_t^{-1}R),\mu_t^{-1}R)$ increases to $H$ as $t\downarrow 0$, we can take $t_{0}=t_{0}(N)\in (0,1]$ such that for any $t\in (0, t_{0}]$,
$
D(N)  \subset B((\tilde{0},\mu_t^{-1}R),\mu_t^{-1}R), 
$
and hence,
\begin{align}\label{eqn:ball lb3}
\P_{(\tilde{0},v)}(\tau_{B((\tilde{0},\mu_t^{-1}R),\mu_t^{-1}R)}^{\widetilde{X}^{(t)}} \leq 1)
\leq \P_{(\tilde{0},v)}(\tau_{D(N)}^{\widetilde{X}^{(t)}}\leq 1).
\end{align}
Combining \eqref{eqn:ball lb2} and \eqref{eqn:ball lb3}, applying the reverse Fatou's lemma on the bounded interval $(0,M\sqrt{d}]$, and using the weak convergence of $\widetilde{X}^{(t)}$ to $Z$ in Lemma \ref{lemma:weakconvergence}, we obtain
\begin{align}\label{eqn:ball lb7}
&\limsup_{t\downarrow 0}\int_{0}^{M\sqrt{d}}\P_{(\tilde{0},v)}(\tau_{B((\tilde{0},\mu_t^{-1}R),\mu_t^{-1}R)}^{\widetilde{X}^{(t)}} \leq 1)\ud v
\le\limsup_{t\downarrow 0}\int_{0}^{M\sqrt{d}}\P_{(\tilde{0},v)}(\tau_{D(N)}^{\widetilde{X}^{(t)}} \leq 1)\ud v \\
&\le \int_{0}^{M\sqrt{d}}\P_{(\tilde{0},v)}(\tau_{D(N)}^{Z} \leq 1)\ud v
<\int_{0}^{M\sqrt{d}}\P_{(\tilde{0},v)}(\tau_{H}^{Z}\leq 1)\ud v+\eps\notag\\
&=\int_{0}^{M\sqrt{d}}\P(\sup_{s\in[0,1]}Y_s\ge 1)\ud v+\eps.\notag
\end{align}

Finally, putting together \eqref{eqn:ball lb1}, \eqref{eqn:ball lb6} and \eqref{eqn:ball lb7} yields
\begin{align*}
\limsup_{t\downarrow 0}\mu_t^{-1}\int_{0}^{a}\P_{(\tilde{0},u)}(\tau^{X}_{B((\tilde{0},R),R)}\leq t)\ud u
<\int_{0}^{M\sqrt{d}}\P(\sup_{s\in[0,1]}Y_s\ge 1)\ud v+\eps
	+ \frac{32d\sqrt{d}}{c_2^4(M-1)^2}.
\end{align*}
Letting $\eps\downarrow 0$ and $M\uparrow \infty$ yields the upper bound
$
\limsup_{t\downarrow 0}\mu_t^{-1}\int_{0}^{a}\P_{(\tilde{0},u)}(\tau^{X}_{B((\tilde{0},R),R)}\le t)\ud u\leq \E[\sup_{s\in[0,1]}Y_s].
$
\end{proof}

Recall the notations $\delta_D(x)$ and $D_a$ defined in \eqref{def:Da}.
For a starting point $x\in D\setminus D_{R/2}$ (so that $\delta_D(x)\le R/2$), let $z_x$ denote the unique point in $\partial D$ such that $|x-z_x|=\delta_{D}(x)$, and let $\mathbf{n}_{z_{x}}=(z_{x}-x)/|z_x-x|$ denote the outward unit normal vector to $\partial D$ at the point $z_x$. 
With the \textit{interior $R$-ball condition} in \eqref{def:interior-exterior} in mind,
let $H_{x}$ denote the unique half-space containing the interior $R$-ball at the point $z_x$ whose normal vector is $\textbf{n}_{z_{x}}$.

\begin{lemma}\label{lemma:cancellation}
Let $R$ be the characteristic radius of the $C^{1,1}$ domain $D$.
Then for any fixed $a\in(0, R/2]$, as $t\downarrow 0$, 
\[
\int_{D\setminus D_{a}}\P_{x}(\tau^{X}_{D}\leq t<\tau^{X}_{H_x})\ud x=o(\mu_t). 
\]
\end{lemma}

\begin{proof}
By \eqref{eqn:vdBD89-b},
\begin{align*}
&\int_{D\setminus D_{a}}\P_{x}(\tau^{X}_{D}\leq t<\tau^{X}_{H_x})\ud x
\leq 2^{d-1}|\partial D|\int_{0}^{a}\P_{(\tilde{0}, u)}(\tau^{X}_{B((\tilde{0},R),R)}  \leq t<\tau^{X}_{H}   )\ud u\\
&=2^{d-1}|\partial D|\left( \int_{0}^{a}\P_{(\tilde{0}, u)}( \tau^{X}_{B((\tilde{0},R),R)} \leq t)\ud u-\int_{0}^{a}\P_{(\tilde{0}, u)}( \tau^{X}_{H} \leq t)\ud u\right).
\end{align*}
Applying Lemmas \ref{lemma:half-space} and \ref{lemma:ball} gives the desired conclusion.  
\end{proof}

Now we are ready to derive the upper bound for Theorem \ref{thm:GaussianMarkov_spectral_high}. 

\begin{proof}[Derivation of the upper bound for Theorem \ref{thm:GaussianMarkov_spectral_high}.]
Let $R$ be the characteristic radius of the $C^{1,1}$ domain $D$.
It follows from \eqref{eqn:vdBD89} that for any $\eps>0$, there exists $a=a(\eps)\in (0,R/2]$ such that 
\begin{align}\label{D-and-Du}
	|\partial D| - \eps < |\partial D_{u}|<|\partial D| +\eps \ \ \textrm{for all} \ u\leq a,
\end{align}
where the notation $D_u$ is defined as in \eqref{def:Da}.
For this particular $a$, note that
\[
|D|-Q_{D}^{X}(t)=\int_{D_{a}}\P_{x}(\tau_{D}^{X}\leq t)\ud x +\int_{D\setminus D_{a}}\P_{x}(\tau_{D}^{X}\leq t)\ud x.
\]
The first term on the right-hand side is negligible when compared to $\mu_t$ due to Lemma \ref{lemma:inside}. 
On the other hand, for any $x\in D$ with $\delta_{D}(x)\le a$ (in other words, $x\in D\setminus D_a$), it follows that 
$
\{\tau_{D}^{X}\leq t\}\subset \{\tau^{X}_{H_x}\leq t\} \cup \{\tau_{D}^{X}\leq t <\tau_{H_x}^{X}\},
$
and therefore,
\begin{align}\label{inequality_split_1}
\int_{D\setminus D_{a}}\P_{x}(\tau_{D}^{X}\leq t)\ud x
\leq \int_{D\setminus D_{a}}\P_{x}(\tau_{H_x}^{X}\leq t)\ud x
 +\int_{D\setminus D_{a}}\P_{x}(\tau_{D}^{X}\leq t <\tau_{H_x}^{X})\ud x.
\end{align} 
The first integral on the right-hand side equals
$\int_{0}^{a}|\partial D_{u}|\cdot\P_{(\tilde{0}, u)}(\tau_{H}^{X}\leq t)\ud u,$
so \eqref{D-and-Du} and Lemma \ref{lemma:half-space} together give
\begin{align}\label{inequality_split_2}
\lim_{t\downarrow 0}\mu_t^{-1}\int_{D\setminus D_{a}}\P_{x}(\tau_{H_x}^{X}\leq t)\ud x=|\partial D|\cdot\E[\sup_{s\in[0,1]}Y_s].
\end{align}
Combining \eqref{inequality_split_1}, \eqref{inequality_split_2}, and Lemma \ref{lemma:cancellation} yields the upper bound
\[
\limsup_{t\downarrow 0}\mu_t^{-1}(|D|-Q_{D}^{X}(t))\leq |\partial D|\cdot\E[\sup_{s\in[0,1]}Y_s],
\]
as desired. 
\end{proof}

Establishing the lower bound for Theorem \ref{thm:GaussianMarkov_spectral_high} requires two additional lemmas,
Lemmas \ref{lemma:outer ball} and \ref{lemma:cancellation2}, which are analogous to Lemmas \ref{lemma:ball} and \ref{lemma:cancellation}, respectively. 
For the lower bound, we take advantage of the \textit{exterior $R$-ball condition} in \eqref{def:interior-exterior}
for the $C^{1,1}$ open set $D$. 
Lemma \ref{lemma:outer ball} considers the first exit time $\tau^{X}_{B((\tilde{0},-R),R)^{c}}$ from the \textit{complement} of the ball $B((\tilde{0},-R),R)$ of radius $R$ centered at the point $(\tilde{0},-R)$, which is located in the lower half-space 
\[
	H'=\{x=(x^1\ldots, x^{d})\in \R^d : x^{d}< 0\}.
\]

\begin{lemma}\label{lemma:outer ball}
Let $R$ be the characteristic radius of the $C^{1,1}$ domain $D$.
Then for any fixed $a\in(0, R/2]$,
\[
\lim_{t\downarrow 0}\mu_t^{-1}\int_{0}^{a}\P_{(\tilde{0},u)}(\tau^{X}_{B((\tilde{0},-R),R)^{c}}\leq t)\ud u=\E[\sup_{s\in[0,1]}Y_s].
\]
\end{lemma}

\begin{proof}
The proof is similar to the proof of Lemma \ref{lemma:ball}, but we provide the details for the reader's convenience. 
Fix $a\in (0,R/2]$.
Since $H\subset B((\tilde{0},-R),R)^{c}$, it follows that $\P_{(\tilde{0},u)}(\tau^{X}_{B((\tilde{0},-R),R)^{c}}\le t)\leq \P_{(\tilde{0},u)}(\tau_{H}^{X}\le t)$, so by Lemma \ref{lemma:half-space}, the upper bound
$
\limsup_{t\downarrow 0}\mu_t^{-1}\int_{0}^{a}\P_{(\tilde{0},u)}(\tau^{X}_{B((\tilde{0},-R),R)^c}\leq t)\ud u\leq \E[\sup_{s\in[0,1]}Y_s]
$
follows.

To derive the lower bound, let $\eps>0$ and recall that $\E[\sup_{s\in[0,1]}Y_s]\le 1<\infty$. There exists an integer $M=M(\eps)>0$ such that
\begin{align}\label{eqn:ex ball lb3}
\int_{0}^{M}\P(\sup_{s\in[0,1]}Y_s> x)\ud x>\E[\sup_{s\in[0,1]}Y_s]-\eps.
\end{align}
As in the proof of Lemma \ref{lemma:half-space}, one can verify that the law of $\tau_{B((\tilde{0},-R),R)^{c}}^{X}$ under $\P_{x}$ is equal to the law of $t\tau_{B((\tilde{0},-\mu_t^{-1}R),\mu_t^{-1}R)^{c}}^{\widetilde{X}^{(t)}}$ under $\P_{\mu_t^{-1}x}$, 
where $\widetilde{X}^{(t)}$ is the scaled process defined in Lemma \ref{lemma:weakconvergence}. 
Hence, by the change of variables $v=\mu_t^{-1}u$,  
\begin{align}\label{eqn:ex ball lb1}
&\int_{0}^{a}\P_{(\tilde{0},u)}(\tau^{X}_{B((\tilde{0},-R),R)^{c}}\leq t)\ud u
=\int_{0}^{a}\P_{(\tilde{0},\mu_t^{-1}u)}(\tau_{B((\tilde{0},-\mu_t^{-1}R),\mu_t^{-1}R)^{c}}^{\widetilde{X}^{(t)}} \leq 1)\ud u\\
&=\mu_t\int_{0}^{a\mu_t^{-1}/2}\P_{(\tilde{0},v)}(\tau_{B((\tilde{0},-\mu_t^{-1}R),\mu_t^{-1}R)^{c}}^{\widetilde{X}^{(t)}} \leq 1)\ud v
\geq \mu_t\int_{0}^{M}\P_{(\tilde{0},v)}(\tau_{B((\tilde{0},-\mu_t^{-1}R),\mu_t^{-1}R)^c}^{\widetilde{X}^{(t)}} \leq 1)\ud v\notag
\end{align}
for all sufficiently small $t>0$ (since $\mu_t^{-1}\to\infty$ as $t\downarrow 0$). 

Now we construct an increasing sequence of domains $E(n)$ by
\[
E(n)=\left\{(\tilde{x}, x^{d})\in \R^{d-1}\times \R: |\tilde{x}|<n,\, -n<x^{d}< 0,\, \cos^{-1}
\Bigl(-\vec{\bf{n}}\cdot \frac{x}{|x|} \Bigr)
<\frac{\pi}{2}-\frac{1}{n} \right\},
\]
where $-\vec{\bf{n}}=(\tilde{0},-1)$ (see Figure \ref{figure_En}). 
\begin{figure}
\begin{center}
\includegraphics[height=3.5cm,clip]{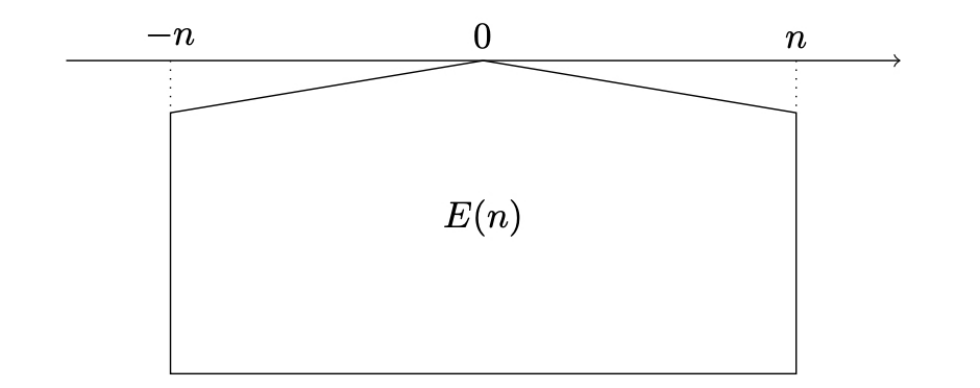}
\caption{The domain $E(n)$ increases to the lower half-space $H'$.}
\label{figure_En}
\end{center}
\end{figure}
For the limiting process $Z$ in Lemma \ref{lemma:weakconvergence}, since $E(n)$ increases to the lower half-space $H'=\{(x^1,\ldots, x^{d}) : x^{d}< 0\}$ as $n\to\infty$, it follows from the monotone convergence theorem that 
\[
\lim_{n\to\infty}\int_{0}^{M}\P_{(\tilde{0},v)}(\tau_{E(n)^c}^{Z}< 1)\ud v=\int_{0}^{M}\P_{(\tilde{0},v)}(\tau_{(H')^{c}}^{Z}< 1)\ud v = 
  \int_{0}^{M}\P_{(\tilde{0},v)}(\tau_{\overline{H}}^{Z}< 1)\ud v.
\]
Since the latter sequence converges increasingly, we may take an integer $N=N(\eps)$ such that 
\begin{align}\label{eqn:ex ball lb2}
\int_{0}^{M}\P_{(\tilde{0},v)}(\tau_{E(N)^{c}}^{Z}< 1)\ud v>\int_{0}^{M}\P_{(\tilde{0},v)}(\tau_{\overline{H}}^{Z}< 1)\ud v - \eps.
\end{align}
Since $B((\tilde{0},-\mu_t^{-1}R),\mu_t^{-1}R)$ increases to $H'$ as $t\downarrow 0$, we can take $t_{0}=t_{0}(N)\in (0,1]$ such that for any $t\in (0, t_{0}]$,
$
E(N)  \subset B((\tilde{0},-\mu_t^{-1}R),\mu_t^{-1}R), 
$
which implies that 
\begin{align}\label{eqn:ex ball lb4}
\P_{(\tilde{0},v)}(\tau_{B((\tilde{0},-\mu_t^{-1}R),\mu_t^{-1}R)^{c}}^{\widetilde{X}^{(t)}} \leq 1)
\geq \P_{(\tilde{0},v)}(\tau_{E(N)^{c}}^{\widetilde{X}^{(t)}}\leq 1).
\end{align}

Combining \eqref{eqn:ex ball lb1}--\eqref{eqn:ex ball lb4} and using Fatou's lemma and the weak convergence of $\widetilde{X}^{(t)}$ to $Z$ in Lemma \ref{lemma:weakconvergence}, we obtain
\begin{align*}
&\liminf_{t\downarrow 0}\mu_t^{-1}\int_{0}^{a}\P_{(\tilde{0},u)}(\tau^{X}_{B((\tilde{0},-R),R)^{c}}\leq t)\ud u
\geq \liminf_{t\downarrow 0}\int_{0}^{M}\P_{(\tilde{0},v)}(\tau_{E(N)^{c}}^{\widetilde{X}^{(t)}} < 1)\ud v\\
&\ge \int_{0}^{M}\P_{(\tilde{0},v)}(\tau_{E(N)^{c}}^{Z}< 1)\ud v
> \int_{0}^{M}\P_{(\tilde{0},v)}(\tau_{\overline{H}}^{Z}< 1)\ud v - \eps
= \int_{0}^{M}\P(\sup_{s\in[0,1)}Y_s> v)\ud v-\eps\\
&= \int_{0}^{M}\P(\sup_{s\in[0,1]}Y_s> v)\ud v-\eps
> \E[\sup_{s\in[0,1]}Y_s]-2\eps,
\end{align*}
where the last two lines follow from the continuity of sample paths of $Y$ and \eqref{eqn:ex ball lb3}.
Since $\eps>0$ is arbitrary, we conclude that 
$
\liminf_{t\downarrow 0}\mu_t^{-1}\int_{0}^{a}\P_{(\tilde{0},u)}(\tau^{X}_{B((\tilde{0},-R),R)^{c}}\leq t)\ud u\geq \E[\sup_{s\in[0,1]}Y_s],
$
as desired. 
\end{proof}

Recall that for a starting point $x\in D\setminus D_{R/2}$, $H_{x}$ denotes the half-space containing the interior $R$-ball at the point $z_x$ and tangent to $\partial D$. Let $\tilde{B_x}$ denote the unique exterior $R$-ball in $D^{c}$ touching the point $z_x$. (Clearly, $H_x\subset \tilde{B_x}^{c}$.) 

\begin{lemma}\label{lemma:cancellation2}
Let $R$ be the characteristic radius of the $C^{1,1}$ domain $D$.
Then for any fixed $a\in(0, R/2]$, as $t\downarrow 0$, 
\[
\int_{D\setminus D_{a}}\P_{x}(\tau_{H_x}^{X}\leq t <\tau_{\tilde{B_x}^{c}}^{X} )\ud x=o(\mu_t). 
\]
\end{lemma}

\begin{proof}
By \eqref{eqn:vdBD89-b},
\begin{align*}
&\int_{D\setminus D_{a}}\P_{x}(\tau_{H_x}^{X}\leq t <\tau_{\tilde{B_x}^{c}}^{X} )\ud x
=\int_0^{a} |\partial D_u| \cdot\P_{(\tilde{0},u)} (\tau_{H}^{X}\leq t <\tau_{B((\tilde{0},-R),R)^{c}}^{X} )\ud u\\
&\le 2^{d-1} |\partial D|\int_0^{a} \P_{(\tilde{0},u)} (\tau_{H}^{X}\leq t <\tau_{B((\tilde{0},-R),R)^{c}}^{X} )\ud u\\
&= 2^{d-1} |\partial D|\left(\int_0^{a} \P_{(\tilde{0},u)} (\tau_{H}^{X}\leq t)\ud u -\int_0^{a} \P_{(\tilde{0},u)} (\tau_{B((\tilde{0},-R),R)^{c}}^{X} \le t)\ud u\right).
\end{align*}
Applying Lemmas \ref{lemma:half-space} and \ref{lemma:outer ball} gives the desired conclusion.  
\end{proof}

\begin{proof}[Derivation of the lower bound for Theorem \ref{thm:GaussianMarkov_spectral_high}.]
Let $R$ be the characteristic radius of the $C^{1,1}$ domain $D$.
As we did in the derivation of the upper bound for Theorem \ref{thm:GaussianMarkov_spectral_high}, for a fixed $\eps>0$, take $a=a(\eps)\in(0, R/2]$ such that \eqref{D-and-Du} holds, which results in \eqref{inequality_split_2}. 

Under $\P_x$ with $x\in D\setminus D_a$, where $D_a$ is defined in \eqref{def:Da}, 
\begin{align*}
\{\tau_{H_x}^{X}\leq t\} 
&\subset \{\tau_{H_x}^{X}\leq t, \tau_{D}^{X} \leq t\} \cup \{\tau_{H_x}^{X}\leq t<\tau_{D}^{X} \}
\subset  \{\tau_{D}^{X} \leq t\} \cup \{\tau_{H_x}^{X}\leq t <\tau_{D}^{X} \}\notag\\
&\subset  \{\tau_{D}^{X} \leq t\} \cup \{\tau_{H_x}^{X}\leq t <\tau_{\tilde{B_x}^{c}}^{X} \}.\notag
\end{align*}
The latter implies that
$
\P_{x}(\tau_{H_x}^{X}\leq t)\leq \P_{x}(\tau_{D}^{X} \leq t) + \P_{x}(\tau_{H_x}^{X}\leq t <\tau_{\tilde{B_x}^{c}}^{X} ).
$
Hence, 
\begin{align*}
|D|-Q^{X}_{D}(t)
&=\int_{D}\P_{x}(\tau^{X}_{D}\leq t)\ud x
\geq \int_{D\setminus D_{a}}\P_{x}(\tau^{X}_{D}\leq t)\ud x\\
&\geq \int_{D\setminus D_{a}}\P_{x}(\tau_{H_x}^{X}\leq t)\ud x-\int_{D\setminus D_{a}}\P_{x}(\tau_{H_x}^{X}\leq t <\tau_{\tilde{B_x}^{c}}^{X} )\ud x.
\end{align*}
Combining \eqref{inequality_split_2} and Lemma \ref{lemma:cancellation2} gives the lower bound
\[
\liminf_{t\downarrow 0}\mu_t^{-1}(|D|-Q_{D}^{X}(t))\geq |\partial D|\cdot\E[\sup_{s\in[0,1]}Y_s],
\]
as desired.
\end{proof}

\vspace{3mm}

\noindent
\textbf{Acknowledgements:} 
The authors are very grateful to the anonymous referees for
their careful reading and detailed comments that significantly helped improve the exposition of the paper and make some statements precise.

  \bibliographystyle{plain} 
\bibliography{KobayashiK}

\end{document}